\documentclass[11pt]{article}
\pdfoutput=1

\title{
On the Hardness of Meaningful Local Guarantees
\\in Nonsmooth Nonconvex Optimization
}

\author{
  Guy Kornowski\thanks{equal contribution}\\Weizmann Institute of Science\\
  \texttt{guy.kornowski@weizmann.ac.il}
  \and
  Swati Padmanabhan\footnotemark[1]\\Massachusetts Institute of Technology\\
  \texttt{pswt@mit.edu}
  \and
  Ohad Shamir\\Weizmann Institute of Science\\
  \texttt{ohad.shamir@weizmann.ac.il}
}

\usepackage[margin=1in]{geometry}

\usepackage[utf8]{inputenc} 
\usepackage[T1]{fontenc}    
\usepackage{amsfonts, amsmath, amsthm, amssymb}       
\usepackage{mlmodern}
\usepackage{nicefrac}      
\usepackage{microtype}     
\usepackage[shortlabels]{enumitem}
\usepackage{fontawesome}

\usepackage[
backend=biber,
backref=true,
natbib=true,
style=numeric,
giveninits=false, 
maxbibnames = 99, 
maxalphanames=8,
maxcitenames=6, 
sorting=ynt
]{biblatex}
\DefineBibliographyStrings{english}{
  backrefpage={cited on page},
  backrefpages={cited on pages}
}

\usepackage[usenames,dvipsnames]{xcolor}
\usepackage{hyperref}       
\usepackage{url}            

\definecolor{otherlightblue}{RGB}{0, 100, 200}
\definecolor{otherblue}{RGB}{0, 50, 100}
\definecolor{othergreen}{RGB}{60, 120, 0}

\hypersetup{
	colorlinks=true,
	pdfpagemode=UseNone,
	citecolor=othergreen,
	linkcolor=otherlightblue,
	urlcolor=magenta
}

\usepackage{thm-restate}
\numberwithin{equation}{section}

\usepackage{float}
\usepackage{algorithm, algorithmicx, algpseudocode}

\usepackage{makecell, booktabs, nicematrix} 
\usepackage{tcolorbox}
\usepackage[font=scriptsize]{caption}
\usepackage{graphicx, subcaption} 
\usepackage{enumitem}

\newlist{thmlist}{enumerate}{1}
\setlist[thmlist]{label=(\roman{thmlisti}),
                  ref=\thethm.(\roman{thmlisti}),
                  noitemsep}

\usepackage[capitalize, nameinlink]{cleveref}

\crefformat{none}{(#2#1#3)}
\crefname{equation}{Equation}{Equations}
\crefformat{equation}{(#2#1#3)}
\crefname{fact}{Fact}{Facts}
\crefname{lemmma}{Lemma}{Lemmas}
\crefname{figure}{Figure}{Figures}
\crefname{defn}{Definition}{Definitions}
\crefname{ineq}{Inequality}{Inequalities}
\crefname{corollary}{Corollary}{Corollaries}
\creflabelformat{ineq}{#2{\upshape(#1)}#3} 
\crefalias{thm}{theorem}
\Crefname{listfact}{Fact}{Facts}

\let\ref\cref

\newcommand{\compresslist}{ 
\setlength{\itemsep}{1pt}
\setlength{\parskip}{0pt}
\setlength{\parsep}{0pt}
}

\theoremstyle{plain}
\newtheorem{theorem}{Theorem}[section]

\newtheorem{claim}[theorem]{Claim}
\newtheorem{defn}[theorem]{Definition}
\newtheorem{definition}[theorem]{Definition}
\newtheorem{fact}[theorem]{Fact}
\newtheorem{lemma}[theorem]{Lemma}
\newtheorem{remark}[theorem]{Remark}

\newlist{propenum}{enumerate}{1} 
\setlist[propenum]{label=(\roman*), ref=\theproposition(\roman*)}
\crefalias{propenumi}{proposition} 
\crefname{prop}{Proposition}{Propositions}

\newlist{factenum}{enumerate}{1} 
\setlist[factenum]{label=(\roman*), ref=\thefact(\roman*)}
\crefalias{factenumi}{fact} 
\crefname{fact}{Fact}{Facts}

\newlist{lemmenum}{enumerate}{1}
\setlist[lemmenum]{label=(\roman*), ref=\thelemma(\roman*)}
\crefalias{lemmenumi}{lemma} 
\crefname{lemma}{Lemma}{Lemmas} 

\newlist{thmenum}{enumerate}{1} 
\setlist[thmenum]{label=(\roman*), ref=\thelemma(\roman*)}
\crefalias{thmenumi}{lemma}

\newcommand{\reals}{\mathbb{R}}
\newcommand{\dist}{\mathrm{dist}}
\newcommand{\NN}{\mathbb{N}}
\newcommand{\OO}{\mathbb{O}}
\newcommand{\Sbb}{\mathbb{S}}
\newcommand{\E}{\mathbb{E}}
\newcommand{\Acal}{\mathcal{A}}

\newcommand{\bzero}{{\mathbf{0}}}
\newcommand{\bx}{\mathbf{x}}
\newcommand{\bw}{\mathbf{w}}
\newcommand{\be}{\mathbf{e}}
\newcommand{\bg}{\mathbf{g}}

\newcommand{\bz}{\mathbf{z}}

\newcommand{\bv}{\mathbf{v}}
\newcommand{\bu}{\mathbf{u}}
\newcommand{\by}{\mathbf{y}}

\global\long\def\intv{\mathbf{\mathbf{I}}}%

\global\long\def\so{\sigma_{1}}%
\global\long\def\st{\sigma_{2}}%
\global\long\def\si{\sigma_{i}}%
\global\long\def\sk{\sigma_{k}}%
\global\long\def\giz{g_{0}^{(i)}}%
\global\long\def\gio{g_{1}^{(i)}}%
\global\long\def\gisi{g_{\sigma_{i}}^{(i)}}%
\global\long\def\goso{g_{\sigma_{1}}^{(1)}}%
\global\long\def\gtst{g_{\sigma_{2}}^{(2)}}%
\global\long\def\gthsth{g_{\sigma_{3}}^{(3)}}%

\global\long\def\pio{\phi_{1}^{(i)}}%
\global\long\def\piz{\phi_{0}^{(i)}}%
\global\long\def\poso{\phi_{\sigma_{1}}^{(1)}}%
\global\long\def\pksk{\phi_{\sk}^{(k)}}%
\global\long\def\pisi{\phi_{\sigma_{i}}^{(i)}}%
\global\long\def\ptst{\phi_{\sigma_{2}}^{(2)}}%
\global\long\def\pNsN{\phi_{\sigma_{N}}^{(N)}}%

\newcommand{\hbNs}{\overline{h}^N_{\sigma}}

\newcommand{\thib}{\theta^{(i)}_{\mathrm{base}}}
\newcommand{\this}{\theta^{(i)}_{\mathrm{shift}}}
\newcommand{\thipb}{\theta^{(i+1)}_{\mathrm{base}}}

\newcommand{\thiNpb}{\theta^{(N+1)}_{\mathrm{base}}}
 
\global\long\def\Po{\Phi^{(1)}}%
\global\long\def\Pt{\Phi^{(2)}}%
\global\long\def\PN{\Phi^{(N)}}%

\global\long\def\rNs{r_{\sigma}^{(N)}}%

\global\long\def\di{\delta^{(i)}}%

\global\long\def\ei{\epsilon^{(i)}}%

\global\long\def\defeq{\stackrel{\textrm{def}}{=}}%
\global\long\def\hf{\tfrac{1}{2}}%

\newcommand{\norm}[1]{\left\|#1\right\|}
\newcommand{\inner}[1]{\left\langle#1\right\rangle}

\newcommand\numberthis{\addtocounter{equation}{1}\tag{\theequation}} 

\newlist{itemizec}{itemize}{2}
\setlist[itemizec,1]{label=\faCaretRight ,wide, parsep= 0.05pt, left = 15pt}

\begin{document}
\maketitle
\begin{abstract}
We study the oracle complexity of nonsmooth nonconvex optimization, with the algorithm assumed to have access only to local function information.
It has been shown by Davis, Drusvyatskiy, and Jiang (2023) that for nonsmooth Lipschitz functions satisfying certain regularity and strictness conditions, perturbed gradient descent converges to local minimizers \emph{asymptotically}. 
Motivated by this result and by other recent algorithmic advances in nonconvex nonsmooth optimization concerning Goldstein stationarity, we consider the question of obtaining a non-asymptotic rate of convergence to local minima for this problem class. 

We provide the following negative answer to this question: Local algorithms acting on regular Lipschitz functions \emph{cannot}, in the worst case, provide meaningful local guarantees in terms of function value in sub-exponential time, even when all near-stationary points are global minima.
This sharply contrasts with the smooth setting, for which it is well-known that standard gradient
methods can do so in a dimension-independent rate.
Our result complements the rich body of work in the 
theoretical computer science
literature
that provide hardness results conditional on conjectures such as $\mathsf{P}\neq\mathsf{NP}$
or
cryptographic assumptions, in that ours holds unconditional of any such assumptions.
\end{abstract}

\newpage

\section{Introduction}\label{sec:introduction}

Nonconvex optimization problems are ubiquitous throughout the computational and applied sciences. Since globally optimizing nonconvex objectives is infeasible in general, optimization theory has long pursued iterative algorithms that find solutions satisfying some \textit{local} optimality guarantees.
For example, given a sufficiently smooth objective $f:\reals^d\to\reals$, 
a folklore result asserts that gradient descent, when applied with a suitable step-size, converges to a stationary point at a dimension-independent rate.
Similarly, the perturbed gradient descent
update
\[
\bx_{t+1}\leftarrow \bx_t-\eta_t\nabla f(\bx_t)+\xi_t,
\numberthis\label{eq:PGD}
\]
where $\xi_t$ is a mean-zero random variable, is known \citep{jin2021nonconvex} to asymptotically converge only to local minimizers of $f$ for suitable choices of $(\eta_t,\xi_t)_{t\in\NN}$. Moreover, under a \emph{``strictness''} assumption~\cite{ge2015escaping}  stating, roughly, that critical points are either sufficiently negatively curved or local minimizers, this convergence is known to have a favorable polynomial rate  nearly independent of the dimension \cite{jin2017escape}.

\looseness=-1Even though the state of affairs is relatively well understood for smooth objectives, many modern applications in machine learning, operations research, and statistics (e.g., deep learning with ReLU activations~\cite{lecun2015deep}, piecewise affine regression~\cite{cui2021modern})
require solving nonconvex problems that are also inherently \textit{nonsmooth}. 
This structure presents several important challenges: As a prominent example, it calls into question the correctness of automatic differentiation with
PyTorch and TensorFlow~\cite{kakade2018provably}. 
Aiming at this regime,  \citet{davis2022proximal} studied (nonsmooth) weakly-convex functions,\footnote{A function $f$ is called weakly-convex if there exists $\rho>0$ such that $\bx\mapsto f(\bx)+\frac{\rho}{2}\norm{\bx}^2$ is convex.} showing that strictness enables proximal methods to asymptotically converge to local minimizers. Subsequently, this convergence was  shown to have a polynomial rate \cite{huang2021escaping,davis2022escaping}, thus extending to a nonsmooth setting what was previously known only for smooth objectives.

However, many prominent contemporary optimization problems such as neural network training fall outside this class of weakly
convex objectives.\footnote{Indeed, this is the case even for a single negated ReLU neuron, namely $z\mapsto-\max\{0,z\}$.} In this regard,
 \citet{davis2023active} 
 proved a vast generalization of the previously mentioned results, asserting that for nonsmooth Lipschitz functions satisfying mild regularity properties and a strictness assumption,
the dynamics dictated by \cref{eq:PGD} asymptotically 
converge only to local minimizers.
Another asymptotic result of similar spirit was also obtained in
\cite{bianchi2023stochastic}.
These developments motivate the following natural question:
\begin{quote}
\centering
\textit{
Is it possible to obtain non-asymptotic convergence guarantees to local minima, 
when optimizing sufficiently regular Lipschitz objectives
that satisfy a strictness property?
}
\end{quote}

A priori, recent advances in nonsmooth nonconvex optimization suggest that there is room for optimism for such finite-time guarantees. Following the work of \citet{zhang2020complexity}, a surge of recent results showed that it is possible to converge, at a dimension-free polynomial rate, to approximate-stationary points in the sense of Goldstein~\cite{goldstein1977optimization} when optimizing Lipschitz functions \cite{tian2022finite,davis2022gradient}. This remains an active area of research, with recent works obtaining finite-time guarantees for optimizing Lipschitz functions in terms of Goldstein-stationarity under a variety of settings such as stochastic \cite{cutkosky2023optimal}, constrained \cite{grimmer2023goldstein}, and zeroth-order optimization \cite{lin2022gradient,kornowski2024algorithm}.

Nevertheless, as our main result (\cref{{thm: stat hardness}}) will soon show, obtaining \textit{any} non-trivial convergence rate in terms of local function decrease is impossible even for strict functions.
In particular, we prove that even under suitable regularity assumptions and the non-existence of non-strict saddles,
any algorithm whatsoever based on local queries will necessarily get stuck, in the worst case, at points at which there is significant local decrease,
unless the number of iterations grows exponentially with the dimension. In fact, this statement holds even under the supposedly easier case in which all approximate-stationary points below some constant (sub)gradient norm are in fact global minima -- which trivially precludes the existence of non-strict saddles.

\subsection{Related Work}
As noted in the introduction, there has been a long line of work in the optimization literature on studying convergence of first-order methods for various classes of nonconvex programs. 
Some early works include those of \citet{nurminskii1974minimization, norkin1986stochastic, ermol1998stochastic, benaim2005stochastic}. 
More recently, \citet{davis2020stochastic} showed the first rigorous convergence guarantees for the stochastic subgradient method on Whitney stratifiable functions by building on techniques from \citet{drusvyatskiy2015curves, duchi2018stochastic}. For
well-behaved problems with $\rho$-weakly convex objective functions, \citet{davis2019stochastic, davis2019proximally}
studied convergence to stationary points of the Moreau envelope
and  provide dimension-free convergence rates for finding these points. As alluded to earlier, the work of Zhang et al~\cite{zhang2020complexity} provided the first finite-time dimension-free guarantees for converging to a Goldstein stationary point of a given Lipschitz function, a result that was subsequently strengthened to hold under standard first-order oracle access~\cite{tian2022finite, davis2022gradient}. Finally, by slightly restricting the class of nonsmooth objectives, \citet{kong2023cost} develop a simple de-randomized version of the algorithm of \cite{zhang2020complexity} with increased, but still dimension-independent, complexity and quantify how
the cost in complexity of optimizing nonsmooth objectives grows with their level of nonconvexity. 

Complementing these algorithmic results, there has also been extensive effort providing hardness results on reaching different solution concepts in nonsmooth nonconvex optimization.  
The computational intractability of globally minimizing a Lipschitz function up to a small
constant tolerance was known since the works of \citet{nemirovskiyudin1983, murty1987some}. 
More recently, Zhang et al~\citep{zhang2020complexity} showed that local, first-order algorithms acting on nonsmooth nonconvex functions cannot attain either small
function error or small gradients: Indeed, approximately-stationary points can be easily ``hidden'' inside some arbitrarily small neighbourhood, which cannot in general be found in a finite number of iterations. Subsequently, \citet{kornowski2022oracle} considered if, for general Lipschitz functions, instead of trying to find approximately-stationary points, the more
relaxed notion of getting \emph{near} an approximately-stationary point is more tractable. Via the construction of a novel  hard function, \cite{kornowski2022oracle} answer this question negatively (for both deterministic and randomized
algorithms). 
This hardness result was subsequently adapted by \citet{tian2021hardness}, for the case of deterministic algorithms, endowing it with Clarke regularity by employing a Huber-type smoothing.
Next, \citet{jordan2023deterministic} provide an improved understanding of achieving Goldstein stationarity by showing that no deterministic algorithm can achieve a dimension-free rate of convergence.
The work of \citet{tian2024no} additionally showed that no deterministic general zero-respecting algorithm for achieving Goldstein stationarity has finite-time complexity.

Our work adds to this line of hardness results: For oracle-based algorithms seeking to achieve local optimality in Lipschitz functions, we prove a lower bound that is exponential in dimension, even when the algorithm is allowed to employ randomization, the function is Clarke regular, and all near-stationary points of the function are in fact global minima.

In a somewhat different direction, it is interesting to compare our result to a rich body of work in the theoretical computer science literature. Some of the earliest such works include those of \citet{sahni1972some, sahni1974computationally}, which showed that global optimization of a general quadratic program is $\mathsf{NP}$-hard, and those by \citet{murty1987some, pardalos1988checking}, which showed that it is $\mathsf{NP}$-hard to test
whether a given point is a local minimizer
 for constrained nonconvex
quadratic programming. The recent work of \citet{ahmadi2022complexityonthe} shows that unless $\mathsf{P}=\mathsf{NP}$, there cannot be a polynomial-time algorithm that finds a point within Euclidean distance $c^n$ (for any constant $c\geq 0$) of a local minimizer of an
$n$-variate quadratic function over a polytope. Additionally, they show that the problem of deciding whether a quadratic function has a local minimizer
over an (unbounded) polyhedron, and that of deciding if a quartic polynomial has a
local minimizer are $\mathsf{NP}$-hard. In the process, \cite{ahmadi2022complexityonthe} answers
a question posed by \citet{pardalos1992open}. Another open problem listed by \cite{pardalos1992open} was recently settled by \citet{fearnley2024complexity}, who showed that  the problem remains hard even if we are
 searching only for a Karush-Kuhn-Tucker (KKT) point. In particular, they show that the quadratic-KKT problem is $\mathsf{CLS}$-complete (a problem class introduced by \cite{daskalakis2011continuous}), which, by another result of \cite{fearnley2022complexity}, is unlikely to have polynomial-time algorithms.

The most important difference of these results from that of ours is that while the above lower bounds rely on conditional hardness assumptions from complexity theory (such as $\mathsf{P}\neq \mathsf{NP}$), our framework of oracle complexity, which reduces optimization to information theoretic notions, enables proving  lower bounds that are entirely unconditional. 
Furthermore, many of these results are stated in terms of hardness of verification, which, in general, does not imply hardness of search --- namely, finding points of interest, as opposed to verifying that a given point is such.
Finally, verification complexity results typically focus on non-Lipschitz polynomials or other function classes, which, as is, do not directly correspond to known complexity upper bounds discussed in our introduction.

\subsection{Preliminaries}\label{sec:prelims}

\paragraph{Notation.}
We let $\NN = \{1, 2, \dotsc, \}$ be the  natural numbers starting from one. We let boldfaced letters (e.g., $\bx$) denote vectors. 
We denote the $d$-dimensional Euclidean space by $\reals^d$ and by $\langle\cdot,\cdot\rangle$ and $\|{}\cdot{}\|$ its associated inner product and norm, respectively. We use
$\bzero_d$ (or simply $\bzero$ when $d$ is clear from context) to denote the zero vector in $\reals^d$ and $\be_1,\be_2,\ldots$ for the standard basis vectors. Given a vector $\bx$, we denote by $x_i$ its $i$-th coordinate, by $\bx_{1:i}:=(x_1,\dots,x_i)$ its truncation, and by $\bar{\bx}=\tfrac{\bx}{\norm{\bx}}$ the normalized vector (assuming $\bx\neq\bzero$).
We use $B(\bx, \delta)$ to denote the closed Euclidean ball of radius $\delta >0$ centered at $\bx$.

\paragraph{Nonsmooth Analysis.} 
We say a function $f:\reals^d\mapsto\reals$ is $L$-Lipschitz if for any $\bx,\by$, we have $|f(\bx) - f(\by)|\leq L \|\bx-\by\|$. 
Recall that by Rademacher's theorem~\cite{rademacher1919partielle}, Lipschitz functions are differentiable almost everywhere (in the sense of Lebesgue). Throughout our paper, we will be working with $L$-Lipschitz functions for some $L$ (we specify our exact set of assumptions in \cref{sec:lower-bound}). We therefore next collect some relevant quantities associated with Lipschitz functions. 

\begin{definition}\label{def:ordinaryDirDer}
    For any Lipschitz function $f:\reals^d \mapsto\reals$, the (ordinary) directional derivative of $f$ at a point $\bx$ in the direction $\bv$ is defined as \[f^\prime(\bx; \bv):= \lim_{t\rightarrow0^+} \frac{f(\bx+t\bv)-f(\bx)}{t}.\] 
\end{definition}
\begin{definition}[\cite{clarke1981generalized, rockafellar1980generalized}]\label{def:genDirDerClarkeSubDiff}
    For any Lipschitz function $f:\reals^d \mapsto\reals$, the generalized directional derivative of $f$ at a point $\bx$ in the direction $\bv$ is defined as \[f^\circ(\bx; \bv):=
    \underset{\by\rightarrow\bx,\,t\rightarrow0^+}{\lim\sup} \frac{f(\by+t\bv)-f(\by)}{t}.\] 
\end{definition}
\noindent The generalized directional derivative leads to the following definition of the Clarke subdifferential. 
\begin{definition}[\citet{clarke1981generalized, clarke1990optimization}]
For any Lipschitz function $f:\reals^d\to\reals$ and point $\bx\in\reals^d$, the Clarke subdifferential of $f$ at $\bx$ is defined as 
    \[\partial f(\bx):= \{\bg\mid \langle \bg, \bv \rangle \leq f^\circ(\bx; \bv),\, \forall \bv\in\reals^d \},\] with each element $\bg$ of this set termed a Clarke subgradient.  
\end{definition}
\noindent The Clarke subdifferential may be  equivalently defined as 
\[
\partial f(\bx):=\mathrm{conv}\{\bg\,:\,\bg=\lim_{n\to\infty}\nabla f(\bx_n),\,\bx_n\to \bx\}~, 
\]
namely, the convex hull of all limit points of $\nabla f(\bx_n)$ over all sequences of differentiable points $\bx_n$ which converge to $\bx$. 
If the function is continuously differentiable at a point or convex, the Clarke subgradient there reduces to the gradient or subgradient in the convex analytic sense, respectively. 
\noindent Equipped with the notation of the Clarke subdifferential, one may define a Clarke regular function. 
\begin{definition}[\cite{clarke1981generalized}; {\cite[Definition $2.3.4$]{clarke1990optimization}}]\label{def:clarkeRegularFn}
A locally Lipschitz function $f:\reals^d\mapsto\reals$ is Clarke regular at $\bx\in\reals^d$ if for every direction $\bv\in \reals^d,$ the ordinary directional derivative $f^\prime(\bx; \bv)$ exists and \[f^\prime(\bx; \bv) = f^\circ(\bx; \bv).\]
We say the function $f:\reals^d\to\reals$ is  regular if it is Clarke regular at all $\bx
\in\reals^d$.
\end{definition}
\noindent Finally, we denote the minimal norm subgradient at a point $\bx$ as
\[
\bar{\partial}f(\bx):=\arg\min\{\norm{\bg}\,:\,\bg\in\partial f(\bx)\}~,
\]
and we say that  $\bx$ is an $\epsilon$-stationary point of $f(\cdot)$ if $\norm{\bar{\partial}f(\bx)}\leq\epsilon$.

\begin{fact}[{\cite[Proposition $2.3.6$]{clarke1990optimization}}]\label{fact:prop_clarke_regularity_calculus}
Let $f$ be Lipschitz near $x$. 
\begin{factenum}
\compresslist{
    \item\label[fact]{item:prop_cont_diff_implies_reg} If $f$ is continuously differentiable 
    at $x$, then $f$ is regular at $x$. 
    \item\label[fact]{item:prop_reg_convex} If $f$ is convex, then $f$ is regular at $x$. 
    \item\label[fact]{item:prop_reg_finsum} A finite linear combination by nonnegative scalars of functions regular at $x$ is regular at $x$. 
}
\end{factenum}
\end{fact}

\noindent Finally, we state the following important facts from subdifferential calculus that we use. 

\begin{fact}[Proposition 2.3.3 and Theorem 2.3.9 \cite{clarke1990optimization}]\label{fact:prop_clarke_subdifferential_calculus}
    We have that $\partial (g_1+g_2)\subseteq \partial g_1 + \partial g_2$, and if $g_1$ is univariate, then $\partial(g_1 \circ g_2)(\mathbf{x})\subseteq \mathrm{conv}\left\{r_1 \mathbf{r}_2: r_1 \in \partial g_1 (g_2(\mathbf{x})), \mathbf{r}_2\in \partial g_2(\mathbf{x})\right\}$. 
\end{fact}

\begin{fact}[{\cite[Theorem 10.29]{rockafellar2009variational}}]\label{fact:regularity_of_composition}
    If $F$ is regular, and $g$ is regular at all $F(x)$, then  $f(x) = g(F(x))$ is regular at all $x$.  
\end{fact}

\paragraph{Local algorithms.}

We consider iterative algorithms that have access to local information at queried points and proceed based on information gathered along these queries \cite{nemirovskiyudin1983}. 
Formally,
we call an oracle \emph{local} if for any point $\bx$ and any two functions $f,g$ that are equal over some neighborhood of $\bx$, the equation $\OO_{f}(\bx)=\OO_{g}(\bx)$ holds.
At every iteration $t\in\NN$, a local algorithm which aims to optimize an unknown objective $f:\reals^d\to\reals$
chooses an iterate $\bx_t\in\reals^d$, receives the local information $\OO_f(\bx_t)$, and proceeds to choose the next iterate $\bx_{t+1}$ as some (possibly random) mapping of all the oracle outputs seen thus far: $(\OO_f(\bx_1),\dots,\OO_f(\bx_t))$.
An important subclass of local algorithms are first-order algorithms, which utilize an oracle of the form $\OO_{f}(\bx)=(f(\bx),\bg_x)$ where $\bg_\bx\in\partial f(\bx)$ is some consistent choice of a subgradient.

\section{Our Main Result}\label{sec:lower-bound}

We now present our main theorem. Put simply, it states that local algorithms acting on regular Lipschitz functions cannot, in the worst case, guarantee meaningful local guarantees in terms of function value in sub-exponential time, even when all near-stationary points are  global minima.

For comparison, recall that for smooth objectives $f:\reals^d\to\reals$ and sufficiently small $\delta>0$,
gradient descent
gets after $T$ steps to a point
$\bx_T\in\reals^d$ such that
$f(\bx_T)\leq \min_{\bz\in B(\bx_T,\delta)}f(\bz)+O(\frac{\delta}{\sqrt{T}})$, namely, a point locally competitive with nearby function values up to a factor which vanishes in a dimension-free manner.
Our main result precludes precisely that when smoothness is absent.

\begin{theorem} \label{thm: stat hardness}
    For any (possibly randomized) local algorithm $\Acal$ and any $T,d\in\NN$, there is a function $f:\reals^d\to\reals$ so that for some absolute constant $c\geq\tfrac{1}{100}$,  
    the following properties hold: 
    \begin{thmenum}
    \compresslist{
        \item The function $f$ is $1$-Lipschitz and Clarke regular, 
        $f(\bzero)-\inf f\leq 2$,
        and all $c$-stationary points of $f$ are global minima.
        \item\label[theorem]{thm:main result part ii} With probability at least $1-2T\exp(-d/36)$, for all $t\in[T]$ and any $\delta\in(0,1]$, the following inequality holds:
\begin{equation} \label{eq: large decrease}
    \min_{\bz\in B(\bx_t,\delta)}f(\bz)
    <f(\bx_t)-\delta c~.
\end{equation}
    }
    \end{thmenum}
\end{theorem}

To contextualize the result in \cref{thm: stat hardness}, we further note that under the stated 1-Lipschitzness condition,
for any $\delta>0$, the local decrease $f(\bx_t)-\min_{\bz\in B(\bx_t,\delta)}f(\bz)$ can be \emph{at most} $\delta$.
Thus, the theorem shows that unless $T\gtrsim\exp(\Omega(d))$, all iterates $\bx_t$ suffer from a nearly-maximal local decrease; as a result, none of these iterates can be regarded as approximate local minima. 
On the other hand, it is clearly the case that with exponential dimension dependence, a trivial grid search algorithm can guarantee getting anywhere (i.e. over a discretization of a bounded domain), and in particular can achieve approximate local optimality somewhere along the algorithm's trajectory.
Hence \cref{thm: stat hardness} can be seen as asserting that nothing substantially better than a trivial strategy is possible.

The proof of \cref{thm: stat hardness} consists of constructing a variant of the function which was previously used to prove a strong lower bound on the complexity of getting near stationary points of Lipschitz functions \cite{kornowski2022oracle}.
Our analysis further reveals that our constructed function   satisfies that all near-stationary points are in fact global minima.
The prior construction by \cite{kornowski2022oracle} does not apply to Clarke regular functions, which is an important consideration for our purposes in two aspects. First, for the sake of interest of the derived result,  the upper bounds in the context of local minimality, as discussed throughout the introduction, crucially rely on this property.\footnote{We remark that \citet{tian2021hardness} provide a Clarke regular variant of this construction, but their construction implies a lower bound only for \emph{deterministic} algorithms --- as a result, this
fails to address the aforementioned algorithms which are based on the perturbed gradient descent dynamics \eqref{eq:PGD}.} Second, Clarke regularity implies the so called ``Lyapunov property'', asserting that the subgradient flow decreases the function value 
proportionally to the subgradient norm
(see \cite{daniilidis2020pathological} for an elaborate discussion on this property and function classes for which it holds).
Therefore, having established a Clarke regular function for which the algorithms' iterates are nowhere near a point with sub-constant subgradient norm, the Lyapunov property further ensures that by tracking the subgradient flow, the local decrease in function value is significant, hence implying our desired lower bound in terms of function value.

\section{Proof of Our Main Result (\cref{thm: stat hardness})}\label{sec:proof}

In this section, we   prove our main result (\cref{thm: stat hardness}). We first state \Cref{lem: basic hard} (deferring its proof to \Cref{sec:proof of lem basic hard}), a technical result on one-dimensional functions that we crucially use for the construction of our hard instance for \cref{thm: stat hardness}. 

\begin{restatable}{proposition}{propOneDimFnProperties} \label{lem: basic hard}
For any $\gamma>0$ and $T\in\NN$, there exists $\rho>0$ so that the following holds:
For any (possibly randomized) local algorithm $\Acal$, there exists a function $\bar{h}:\reals\to[2,\infty)$ such that
\begin{lemmenum}
\compresslist{
    \item\label[proposition]{lem: basic hard i} $\bar{h}$ is $1$-Lipschitz, convex, and satisfies $\bar{h}(0)\leq 3.$ 
    \item\label[proposition]{lem: basic hard ii} $\bar{h}$ has a unique minimizer $x^*\in(0,1)$, and
    $\forall x\neq x^*:\left|\bar{\partial}\bar{h}(x)\right|\geq\tfrac{1}{8}$.
    \item\label[proposition]{lem: basic hard iii} $\Pr_\Acal[\exists t\in[T]:
    |x_t-x^*|\leq\rho
    ]<\gamma$.
}
\end{lemmenum}
\end{restatable}
\noindent The above proposition considers the class of Lipschitz, convex functions bounded from below and with a certain minimum slope at all points that are not the function minimizers. Then, for any randomized local algorithm, there exists a function in  this function class, such that, with high probability, the said algorithm cannot reach its local minimum. 

 By embedding the one-dimensional hard instance of \Cref{lem: basic hard} into higher dimensions, a simple reduction enables us to extend \Cref{lem: basic hard}  to functions beyond merely one dimension.

\begin{lemma} \label{lem: d>=2 hard}
For any $\gamma>0,~T\in\NN$, and $~d\geq 2$, there exists $\rho>0$ (which depends only on $\gamma,T$) so that the following holds:
For any (possibly randomized) local algorithm $\Acal$, there exists a function $\bar{h}:\reals\to[1,\infty)$ 
satisfying the following properties. 
\begin{lemmenum}
\compresslist{
    \item\label[lemma]{item:lipsch_conv_hzero} $\bar{h}$ is $\tfrac{1}{2}$-Lipschitz, convex, and satisfies $\bar{h}(0)\leq \tfrac{3}{2}$. 
    \item\label[lemma]{item:unique_min_min_slope} $\bar{h}$ has a unique minimizer 
    $x^*\in(0,1)$, and
    $\forall x\neq x^*:\left|\bar{\partial}\bar{h}(x)\right|\geq\tfrac{1}{16}$.
    \item\label[lemma]{lemma_item_prob_premain_proof} When applying $\Acal$ to $h(\bx):=\tfrac{1}{32}\norm{\bx_{1:d-1}}+\bar{h}(x_d)$, we have, for $\bx^*=(\bzero_{d-1},x^*)$,   that
    \[
    \Pr_\Acal[\exists t\in[T]:
     \norm{\bx_t-\bx^*}\leq\rho
    ]<\gamma~.
    \]
}
\end{lemmenum}

\end{lemma}
\begin{proof}[Proof of \Cref{lem: d>=2 hard}] 
Since our goal is to provide a lower bound 
when applying the local algorithm $\Acal$, we can assume without loss of generality that $\Acal$ has access to an even stronger oracle of the form 
\[
\overline{\OO}(\bx)=\left(\left\{h(\bz_{1:d-1},x_d)\mid \bz_{1:d-1}\in\reals^{d-1}\right\},\OO_{h_{\mathrm{1d}}}(x_d)\right),
\]
where $h_{\mathrm{1d}}$  is a one-dimensional function we will soon choose. Note that oracle $\overline{\OO}$ as defined here provides a full description of the function $h$ over the affine subspace $\{\bz\mid z_d=x_d\}$ in addition to the local information with respect to the last coordinate.\footnote{As an example, in the canonical case of a (sub)gradient oracle, $\overline{\OO}$ provides the partial derivatives with respect to the first $d-1$ coordinates at all points, while  revealing the partial derivative only with respect to the last coordinate at the queried point.}
Moreover, given the algorithm $\Acal$ with such an oracle, one can simulate a local algorithm $\Acal'$ which optimizes
the one-dimensional function $h_{\mathrm{1d}}$ by restricting to the $d$'th coordinates $((\bx_{t})_d)_{t\in\NN}$ of the iterates $(\bx_t)_{t\in \NN}$.

We let $h_{\mathrm{1d}}$ be the one-dimensional function given by \Cref{lem: basic hard} when applied to $\gamma,T,\Acal'$. With this choice of  $h_{\mathrm{1d}}$, we let  $\bar{h}=\frac{1}{2}h_{\mathrm{1d}}$. Then, \Cref{item:lipsch_conv_hzero}  and \Cref{item:unique_min_min_slope} are immediate by \cref{lem: basic hard i} and \cref{lem: basic hard ii}. We also have $\bar{h}:\reals\mapsto[1,\infty)$ from the range $[2,\infty)$ of $h_{\mathrm{1d}}$  from \Cref{lem: basic hard}.  Finally, by combining the fact that  $\Acal$ can simulate $\Acal'$ and using \cref{lem: basic hard iii} for $\Acal'$, we have the following inequality, which establishes  \Cref{lemma_item_prob_premain_proof}: 
\[
\Pr_\Acal[\exists t\in[T]:\norm{\bx_t-\bx^*}\leq\rho]
\leq \Pr_{\Acal'}[\exists t\in[T]:|(\bx_t)_d-x^*|\leq\rho]
< \gamma~.
\] 
\end{proof}

\noindent For our subsequent proofs, we use the setup described next, in \cref{def:fwu}. 

\begin{definition}\label{def:fwu}
We let $\Acal$ be a local algorithm, $T,d\geq 2$,
and set $\gamma:=T\exp(-d/36)$. We denote by ${h}:\reals^d\to[1,\infty),~\bx^*\in\reals^d,~\rho>0$ the function, minimizer, and positive constant, respectively, given by \cref{lemma_item_prob_premain_proof} when applied to $\Acal,T,d$.
Given any $\bw\in\reals^d\setminus\{\bzero\}$ and $\mu>0$, we denote $\bar{\bw}:=\tfrac{\bw}{\norm{\bw}}$ and construct 
\[
f_{\bw,\mu}(\bx):=\max\left\{h(\bx)-\sigma_{\mu}( q(\bx-\bx^*)),0\right\},\numberthis\label{eq:def_fwu}
\]
where  $q:\reals^d\to\reals$ 
 and $\sigma_\mu: \reals\to\reals$ are defined as follows:
\[ 
q(\bx) := \inner{\bar{\bw},\bx+\bw}-\tfrac{1}{2}\norm{\bx+\bw}\text{ and } \sigma_\mu(z):=\begin{cases}
    0, & z\leq 0 \\
    \tfrac{z^2}{8\mu}, & z\in(0,\mu] \\
    \tfrac{z}{4}-\tfrac{\mu}{8}, & z>\mu
\end{cases}. \numberthis\label{def:q_and_sigma_mu}
\]
\end{definition}

\noindent We need the following technical result about $f_{\bw, \mu}$, the proof of which we defer to \Cref{sec:proof of lem with many cases}.

\begin{restatable}{lemma}{lemmaWithManyCases} \label{lem: choice of w}
    For the setup in \cref{def:fwu} and $\bw$ such that $\bw\perp\be_d$ and $\|\bw\|=1000\mu$, 
    the following hold:
    \begin{lemmenum}
    \compresslist{
        \item\label[lemma]{lem: choice of w i} $f_{\bw,\mu}$ is non-negative, $1$-Lipschitz, and Clarke regular. 
        \item\label[lemma]{lem: choice of w ii} For $c= \tfrac{1}{100}$, any $c$-stationary point $\bx$ of $f_{\bw,\mu}$ satisfies $f_{\bw,\mu}(\bx)=0$. In particular, any $c$-stationary point of $f_{\bw,\mu}$ is a global minimum. 
        \item\label[lemma]{lem: choice of w iii} There exist $\bw\in\reals^d,~\mu>0$  
        such that applying $\Acal$ to $f_{\bw,\mu}$ satisfies
        \[
        \Pr_\Acal[\exists t\in[T]:f_{\bw,\mu}(\bx_t)<1]<2\gamma~.
        \] 
    }
    \end{lemmenum}
\end{restatable}

\begin{remark}\label{rem:choice_of_w}
Following \cref{lem: choice of w iii}, we set $\bw,\mu$ so that $\Pr_\Acal[\exists t\in[T]:f_{\bw,\mu}(\bx_t)<1]<2\gamma$ holds. For notational brevity, from hereon, we let $f=f_{\bw,\mu}$. 
\end{remark}
We already see by \cref{lem: choice of w i} that $f$ satisfies the regularity assumptions required by \cref{thm: stat hardness}. We therefore complete the proof by showing that it satisfies hardness in terms of getting near stationary points.

\begin{lemma} \label{lem: not stat} For the setup in \cref{def:fwu}, for any  $c= \tfrac{1}{100}$, we have that 
\[\Pr_{\Acal}\left[\exists t\in[T]:\bx_t\text{ is of distance $<1$ to a $c$-stationary point of $f$}\right]\leq 2T\exp(-d/36)
.\] 
\end{lemma}

\begin{proof}[Proof of \Cref{lem: not stat}]
By \cref{lem: choice of w}, with probability at least $1-2\gamma:\min_{t\in[T]}f(\bx_t)\geq1$, while for any $c$-stationary point $\bx:f(\bx)=0$. Under this probable event, recalling that $f$ is $1$-Lipschitz, we get that $(\bx_t)_{t=1}^{T}$ must be of distance of at least $1$ from any $c$-stationary point of $f$. Finally, we complete the proof by recalling that $\gamma:=T\exp(-d/36)$ in \cref{def:fwu}.
\end{proof}

We finalize by showing that if an iterate is far away from any approximate-stationary point, then it cannot be an approximate local minimum.

\begin{lemma} \label{lem: stat to min}
    Let $f$ be a Clarke regular function, 
    and suppose $\bx\in\reals^d,c_\delta>0$ are such that $B(\bx,c_\delta)$ does not contain any $c_\epsilon$-stationary point. Then for every $\delta\leq c_\delta:$
    \[ 
    \min_{\bz\in B(\bx,\delta)}f(\bz)
    <f(\bx)-\delta c_{\epsilon}~.
    \]
\end{lemma}

\begin{proof}[Proof of \cref{lem: stat to min}]
    Consider the subgradient flow $\bx(0)=\bx,~\frac{d\bx(t)}{dt}= -\frac{\bar{\partial} f(\bx(t))}{\norm{\bar{\partial} f(\bx(t))}}$. Note that the flow is well defined throughout $t\in[0,c_{\delta}]$ since it has unit speed, hence $\bx(t)\in B(\bx,t)\subseteq B(\bx,c_\delta)$ and by assumption $B(\bx,c_\delta)$ does not contain any point with zero gradient. Defining $\phi(t):=f(\bx(t))$, we get by the chain rule
    \[
    \frac{d\phi(t)}{dt}=\bar{\partial} f(\bx(t))\cdot\frac{d\bx(t)}{dt}=-\norm{\bar{\partial} f(\bx(t))},
    \]
    thus
    \[
    \min_{\bz\in B(\bx,\delta)}f(\bz)-f(\bx)\leq
    f(\bx(\delta))-f(\bx)
    =\phi(\delta)-\phi(0)=-\int_{0}^{\delta}\norm{\bar{\partial} f(\bx(t))}dt
    <-\delta c_{\epsilon}~.
    \]
\end{proof}

\begin{proof}[Proof of \cref{thm: stat hardness}]  To prove our main result of this paper, the function we consider is $f = f_{\bw,\mu}$ as defined in \cref{def:fwu}, with $\bw$ and $\mu$ chosen so as to have  \cref{lem: choice of w iii} be satisfied.  We now state where we proved all its claimed properties. 
We showed in \cref{{lem: choice of w i}} that this $f$ is $1$-Lipschitz and Clarke regular. We also showed in \cref{{lem: choice of w i}} that $f$ is non-negative, which implies $\inf f\geq 0$.
Next, from \cref{eq:def_fwu}, we have that $f(\mathbf{0}) \leq h(\mathbf{0}) = \bar{h}(0) \leq 3/2$ 
, where we used \cref{{item:lipsch_conv_hzero}} in the final step. This shows $f(\mathbf{0})-\inf f\leq 2$ as claimed. Similarly, we showed in \cref{{lem: choice of w ii}} that any $c$-stationary point of $f$ is a global minimum for  $c=\tfrac{1}{100}$.  
To show \cref{thm:main result part ii}, we combine  \cref{lem: not stat} with \cref{lem: stat to min} (using $c_\epsilon=c,~c_\delta=1$) to complete the proof. 
\end{proof}

\section{Proof of \Cref{lem: basic hard}}\label{sec:1d_construction}
\noindent The goal of this section is to prove \Cref{lem: basic hard}, which we first recall below. 

\propOneDimFnProperties*

Before proving this result (i.e., constructing the function $\overline{h}$), we  describe our high-level idea, followed by  definitions of the components that constitute $\overline{h}$; our proof appears in \Cref{sec:proof of lem basic hard}. 

\subsection{Intuition for \Cref{lem: basic hard}}  Our function $\overline{h}$ is essentially a translated version of $\rNs$ defined in \cref{eq:construction_of_rN}. As can be seen from this defining expression, $\rNs$ is a \emph{random} piecewise-affine scalar-valued function defined over the entire real line. The key functions that make up $\rNs$ are $\pisi$ (\cref{def:phi_and_Phi}), $\gisi$ (\cref{{def:g_i_zero_and_g_i_one}}), and $\Phi^{(i)}$ (\cref{{def:new_def_Phi}}). 

Crucially, the randomness in the parameter  $\sigma\sim \mathrm{Unif}\{0, 1\}^N$ of $\rNs$ determines the both the  intervals and affine functions that constitute $\rNs$. The intervals and their corresponding functions are carefully chosen so as to ensure continuity of $\rNs$: in particular, within  each affine function is baked in a term of the form $\left( \poso \circ\ptst \circ \hdots\right)^{-1}$ that leads to appropriate cancellation at the endpoints of intervals of definition; see, e.g., \cref{eq:first-lipsc-and-cvx}
and \cref{eq:second-lipsc-and-cvx}.  The Lipschitzness  of $\rNs$ follows  from that of each of its pieces, which in turn follows from the chain rule applied to the composition functions making up each of these pieces and \cref{{bounds_on_slope}}, where $\gisi$ is proven Lipschitz. The same line of reasoning also yields a \emph{lower} bound on the slope of $\rNs$. The proof of convexity of $\rNs$ is made simple by our parametrization of it in terms of $\theta^{(i)}$:  the slope of the constituent affine functions of $\rNs$ varies as $\cot(\theta^{(i)})$. The monotonicity of the cotangent function and our chosen range of $\theta^{(i)}$ then  imply that the slope of $\rNs$ increases as we traverse the real line from left to right, thus immediately giving us the desired convexity. We now proceed to state these technical details with the goal of proving \Cref{lem: basic hard}. The proof is in \Cref{{sec:proof of lem basic hard}}.  

\subsection{Technical Details for \Cref{lem: basic hard}}
\begin{defn}\label{def:thetas}
    For $i\in \mathbb{N}$, we define a sequence of angles $\{\thib\}$ and $\{\this\}$  that satisfies \[ \theta^{(1)}_{\mathrm{base}} = \arctan(1), \quad \theta^{(i)}_{\mathrm{shift}} = \frac{\arctan(8)-\theta^{(1)}_{\mathrm{base}}}{2^i}, \quad \thipb = \thib + \this.\numberthis\label{eq:evolution_of_theta}\] For $\thib$ and $\this$, now define the quantities $\ei$ and $\di$ as follows: 
    \[ \ei = 1  - \frac{3}{2} \cdot \left( \frac{1}{2 \tan(\thib + \this) + \tan(\thib)} \right), \quad \di = \frac{1}{2} \cdot \left(\frac{\tan(\thib + \this) - \tan(\thib)}{2\cdot\tan(\thib + \this) + \tan(\thib)}\right).\numberthis\label{eq:def_eps_parametrized_by_theta}
    \]
\end{defn}
\begin{claim}\label{claim:delta_eps_pos}
The $\di$ and $\ei$ and angles $\thib$ from \cref{def:thetas} satisfy, for all $i \in \mathbb{N}$, that \[ \ei >0, \quad 0 < \di \leq\tfrac{7}{32}, \quad\text{ and }\quad \thib\in [\arctan(1), \arctan(8)].\]
\end{claim}
\begin{proof}First, observe that $\thib$ is monotonically increasing in $i$ (as seen from \cref{eq:evolution_of_theta}). Hence, for all $i\in \mathbb{N}$, we have $\thib \geq \theta^{(1)} = \arctan(1)$, which finishes one part of the claim. To state a lower bound on $\ei$ defined in \cref{eq:def_eps_parametrized_by_theta}, we observe that \[2\tan(\thib + \this) + \tan(\thib) \geq 3\tan(\thib) \geq 3\tan(\theta^{(1)}_{\mathrm{base}}) = 3 \tan(\arctan(1)) = 3 ,\] where the first step used the monotoniticity of the tangent function, the second step used the fact that $\thib$ is monotonically increasing (as seen from \cref{eq:evolution_of_theta}), and the final step is by evaluation. Therefore, we have that \[\ei \geq 1- \frac{3}{2 \times 3} > 0,\] which finishes the proof of the claim of positive $\ei$ for all $i\in \mathbb{N}$. To see the bounds on $\di$, we first note that by monotonicity of the tangent function, $\tan(\thib + \this) > \tan(\thib)$, so $\di >0$. For the upper bound, we note that \[\di = \frac{1}{2} \cdot \left(\frac{\tan(\thib + \this) - \tan(\thib)}{2\cdot\tan(\thib + \this) + \tan(\thib)}\right) \leq \frac{1}{2}\cdot \frac{\tan(\thib + \this)-1}{2\tan(\thib + \this)},\] where the first step is because $\thib \geq \arctan(1)$ for all $i$. Finally, noting that $\thib\leq \arctan(8)$ for all $i$,  we can further bound the term above as $\di \leq \tfrac{1}{4} - \tfrac{1}{32} = \tfrac{7}{32},$ as claimed. 
\end{proof}

\begin{defn}
\label{def:phi_and_Phi} 
For $i\in \mathbb{N}$, we use the $\di$ from \cref{def:thetas} 
 to define the  functions $\piz$ 
and $\pio$ as follows. 
\begin{equation}\label{eq:phi_i_zero_and_one}
\begin{aligned}
\piz:[0,1]\rightarrow\left[\hf-2\di,\hf-\di\right]\subseteq(0,1),\text{ where } \piz(x)=\di\cdot x+\hf-2\di,\\
\pio:[0,1]\rightarrow\left[\hf+\di,\hf+2\di\right]\subseteq(0,1),\text{ where }\pio(x)=\di\cdot x+\hf+\di.
\end{aligned}
\end{equation} These are the unique affine maps with positive derivatives mapping $[0, 1]$ to their respective ranges. Our assertion that $[\hf-2\di, \hf-\di]\subseteq(0,1)$ and $[\hf+\di,\hf+2\di]\subseteq(0,1)$ is justified by $0<\di<\tfrac{1}{4}$ from \cref{claim:delta_eps_pos}. 
We use the functions in \cref{eq:phi_i_zero_and_one} to define the interval 
\[
\intv_{\so\st\dotsc\sk}\defeq\poso\circ\dotsc\circ\pksk(0,1)\subseteq(0,1).
\]
We denote the left and right end point of $\intv_{\so\st\dotsc\sk}$
as 
\[\inf\intv_{\so\st\dotsc\sk}=\poso\circ\ptst\circ\dotsc\circ\pksk(0)
\text{ and } \sup\intv_{\so\st\dotsc\sk}=\poso\circ\ptst\circ\dotsc\circ\pksk(1).\numberthis\label{eq:interval_in_terms_of_phi_chain_at_zero}\]
Furthermore, we note the following definitions of the specific intervals $\intv_{0}$ and
$\intv_{1}:$ 
\[
\intv_{0}=\phi_{0}^{(1)}(0,1),\text{ and }\intv_{1}=\phi_{1}^{(1)}(0,1).
\]
\end{defn}
\begin{lemma}\label{lem:interval_properties}The functions defined in \Cref{def:phi_and_Phi} satisfy the following properties. 
\begin{lemmenum}
\compresslist{ 
\item\label[lemma]{asdf} For any $\ell<k$, and any $\so,\st,\dotsc,\sigma_{\ell},\dotsc,\sk\in\left\{ 0,1\right\} ,$
we have that $\intv_{\so\dotsc\sigma_{\ell}}\supseteq\intv_{\so\dotsc\sigma_{\ell}\dotsc\sk}.$
\item\label[lemma]{chain_phi_increasing} For any $i\geq1$, the function $\poso\circ\dotsc\circ\phi_{\sigma_{i}}^{(i)}:(0,1)\mapsto\intv_{\so\dotsc\sigma_{i}}$
is non-decreasing. 
\item\label[lemma]{unique_chains_nonintersecting_intervals} If $(\so\dotsc\sigma_{k-1})\neq(\sigma_{1}^{\prime}\dotsc\sigma_{k-1}^{\prime})$
then for all $\sk,\sigma_{k}^{\prime}\in\left\{ 0,1\right\} :$ $\intv_{\so\dotsc\sk}\cap\intv_{\sigma_{1}^{\prime}\dotsc\sigma_{k}^{\prime}}=\emptyset.$ 
\item\label[lemma]{dist_inv_N_less_del_implies_x_in_intv_k} Let $k = \tfrac{1}{4}\sqrt{\log\left(\tfrac{1}{\rho}\right)}$.  
Then $\textrm{dist}(x, \intv_{\so\dotsc\sigma_N}) \leq\rho$ implies $x\in \intv_{\so \dotsc \sk}$. 
}
\end{lemmenum}
\end{lemma}
\begin{proof}[Proof of \Cref{lem:interval_properties}] We prove each of the parts below.
\paragraph{Proof of \Cref{asdf}.} The claim follows by the following observation: \[ \intv_{\so\dotsc\sigma_{\ell}\dotsc\sk} = \poso\circ\dotsc\circ\phi^{(\ell)}_{\sigma_{\ell}}(\phi^{(\ell+1)}_{\sigma_{\ell+1}}\circ\dotsc\circ\pksk(0,1))\subseteq \poso\circ\dotsc\circ\phi^{(\ell)}_{\sigma_{\ell}}(0,1)=\intv_{\so\dotsc\sigma_{\ell}}.\] 

\paragraph{Proof of \Cref{chain_phi_increasing}.} Observe that from \Cref{def:phi_and_Phi}, the function $\poso\circ\ptst\circ\dotsc\pisi$ is a composition of affine functions and is therefore itself affine; furthermore, by applying the chain rule, one may infer that the derivative of a composition of affine functions equals the product of the derivatives of the individual composing functions,
which in our case is the product $\delta^{(1)}\cdot\delta^{(2)}\dotsc \delta^{(i)}$, which is positive since each $\delta^{(i)}$ is as well.  

\paragraph{Proof of \Cref{unique_chains_nonintersecting_intervals}.} Suppose
$i\leq k-1$ be the minimal index for which $\sigma_{i}\ne\sigma_{i}^{\prime}$
and assume, without loss of generality, that $\sigma_{i}=0$ and $\sigma_{i}^{\prime}=1.$
Consider a point $x$ satisfying $x\in\intv_{\si\dotsc\sk}.$ In other words,
for some $u\in(0,1),$ we have that $x = \poso \circ \dotsc \circ \pksk(0)$. Then, we have: 
\begin{align*}
\poso\circ\dotsc\circ\pksk(u)
 & \leq\sup_{u\in(0,1)}\poso\circ\dotsc\circ\phi_{\sigma_{i-1}}^{(i-1)}\circ\phi_{0}^{(i)}\circ\phi_{\sigma_{i+1}}^{(i+1)}\dotsc\circ\pksk(u)\\
 & \leq\sup_{u\in(0,1)}\poso\circ\dotsc\circ\phi_{\sigma_{i-1}}^{(i-1)}\circ\phi_{0}^{(i)}(u)\\
 & <\poso\circ\dotsc\circ\phi_{\sigma_{i-1}}^{(i-1)}\left(\hf\right)\\
 & \leq\inf_{u\in(0,1)}\poso\circ\dotsc\circ\phi_{\sigma_{i-1}}^{(i-1)}\circ\phi_{1}^{(i)}(u)\\
 & =\inf_{u\in(0,1)}\poso\circ\dotsc\circ\phi_{\sigma_{i-1}}^{(i-1)}\circ\phi_{\sigma_{i}^{\prime}}^{(i)}(u)\\
 & \leq\inf_{u\in(0,1)}\phi^{(1)}_{\sigma^\prime_1}\circ\dotsc\circ\phi_{\sigma^\prime_{i-1}}^{(i-1)}\circ\phi_{\sigma_{i}^{\prime}}^{(i)}\circ\phi_{\sigma_{i+1}^{\prime}}^{(i+1)}\circ\dotsc\phi_{\sigma^\prime_{k}}^{(k)}(u)
 .\numberthis\label[ineq]{ineq:x_le_inf_intv}
\end{align*}
 where the first step is by taking the largest value over all feasible
$u;$ the second step is because $\phi_{\sigma_{i+1}}^{(i+1)}\dotsc\circ\pksk(u)\subseteq(0,1)$
implies we are simply maximimizing over a larger set of arguments
of $\phi_{0}^{(i)}$; the third step uses that by \cref{claim:delta_eps_pos},
we have $\phi_{0}^{(i)}(0)=\hf-2\di\in (0, \hf)$ and further that $\poso\circ\dotsc\circ\phi_{\sigma_{i-1}}^{(i-1)}$
is non-decreasing; the fourth step again uses monotonicity of $\poso\circ\dotsc\circ\phi_{\sigma_{i-1}}^{(i-1)}$
and the fact that $\hf<\phi_{1}^{(i)}(0)=\hf+\di$; the sixth step
uses the fact that we are minimizing over a smaller set as well as replaces, in the first $i-1$ terms of the composition, $\phi^{(j)}_{\sigma_j}$ by $\phi^{(j)}_{\sigma^\prime_j}$. In conclusion, from \cref{ineq:x_le_inf_intv}, 
we have that $x\notin \inf_{u\in(0,1)}\poso\circ\dotsc\circ\phi_{\sigma^\prime_{k}}^{(k)}(u)$ and hence $x\notin \intv_{\sigma^\prime_1\dotsc\sigma^\prime_{k}},$ which finishes the proof of  $\intv_{\so\dotsc\sk}\cap\intv_{\sigma_{1}^{\prime}\dotsc\sigma_{k}^{\prime}}=\emptyset$. 

\paragraph{Proof of \Cref{dist_inv_N_less_del_implies_x_in_intv_k}.}  Recall from \Cref{asdf} that $\intv_{\so\hdots\sigma_N}\subseteq \intv_{\so\hdots\sk}$, and note that for any $i \in \mathbb{N}$: 
\begin{align*} 
\inf \intv_{\so\dotsc\sigma_{i+1}} - \inf \intv_{\so\dotsc\sigma_i} = \poso\circ\dotsc\pisi\circ\phi^{(i)}_{\sigma_{i+1}}(0) - \poso\circ\dotsc\circ\pisi(0) 
            = \prod_{j=1}^i (\phi^{(j)}_{\sigma_j})^\prime \cdot (\phi^{(i+1)}_{\sigma_{i+1}}(0) - 0).
\end{align*} Next, recall from \Cref{eq:phi_i_zero_and_one} that  for every $j\in \mathbb{N}$, the slope $(\phi^{(j)}_{\sigma_j})^\prime = \delta^{(j)}$. We get a lower bound on the product $\prod_{j=1}^i \delta^{(j)}$ as follows. 
\[ \di =\frac{1}{2}\cdot \frac{\tan(\thib+\this)-\tan(\thib)}{2\cdot\tan(\thib+\this)+\tan(\thib)}\geq\frac{\tan(\this)}{12} = \frac{1}{12}\cdot\tan\left(\frac{\arctan(8)-\arctan(1)}{2^i}\right),\numberthis\label[ineq]{eq:delta_simplification}\] where 
the first inequality is due to the identity $\tan(\thib +\this) = \frac{\tan(\thib) + \tan(\this)}{1-\tan(\thib)\tan(\this)}$ and simplifying via the facts that $\tan(\thib) \geq 0$, $\tan(\this)\geq 0$, all angles $\thib\in[\arctan(1),\arctan(8)]$ and that since
the tan function is monotonically increasing in $(0,\pi/2)$, we have
$\tan(\thib+\this)\geq\tan(\thib)$.
We can therefore continue the lower bound on $\prod_{j=1}^i (\phi^{(j)}_{\sigma_j})^\prime \cdot (\phi^{(i+1)}_{\sigma_{i+1}}(0) - 0)$ as follows. 
\begin{align*}
\prod_{j=1}^i (\phi^{(j)}_{\sigma_j})^\prime \cdot (\phi^{(i)}_{\sigma_{i+1}}(0) - 0)
&\geq \frac{1}{16}\cdot\prod_{j=1}^{i}\tan\left(\frac{\arctan(8)-\arctan(1)}{2^{j}}\right)\\ 
&\geq \frac{1}{16}\cdot\prod_{j=1}^i \left( \frac{\arctan(8)-\arctan(1)}{2^{j}} \right)\\ 
&\geq (\arctan(8)-\arctan(1))^{i}\cdot \frac{1}{2^{2i^2+4}}, \numberthis\label[ineq]{delta_lower_bound_arctan_diff}
\end{align*} where the first step is by \Cref{eq:delta_simplification} and lower bounding $\phi^{(i+1)}_{\sigma_{i+1}}(0) = \tfrac{1}{2}-2\di\geq \tfrac{1}{2}- 2\cdot\tfrac{7}{32}\geq \tfrac{1}{16}$ from \Cref{def:phi_and_Phi} and \Cref{claim:delta_eps_pos},
and the second step is by lower bounding each term $\tan(x)$ of the product via $\tan(x)\geq x$, which in turn follows from the power series expansion of the tangent function.
Therefore, we may use \cref{delta_lower_bound_arctan_diff} as follows:
\begin{align*}
    \inf \intv_{\so\dotsc\sigma_N}-\inf \intv_{\so\dotsc\sk}&= \sum_{i = k}^{N-1}(\inf \intv_{\so\dotsc\sigma_{i+1}} - \inf \intv_{\so\dotsc\sigma_i}) \\ 
    &\geq\sum_{i=k}^N \frac{(\arctan(8)-\arctan(1))^{i}}{2^{2i^2+4}} \\ 
    &\geq\frac{(\arctan(8)-\arctan(1))^{k}}{2^{2k^2+4}}\\ 
&\geq \frac{1}{2^{2k^2+k+4}} > \rho, 
\end{align*}
where in the second inequality, we dropped all but the first term (valid since all the terms are non-negative), in the penultimate inequality, we plugged in a lower bound for $\arctan(8)-\arctan(1)$, and the final inequality holds for the choice of $k= \tfrac{1}{4}\sqrt{\log\left(\tfrac{1}{\rho}\right)}$. We  conclude $\inf\intv_{\so\dotsc\sk}<\inf\intv_{\so\dotsc\sigma_N}-\rho$ and analogously $\sup\intv_{\so\dotsc\sk}>\sup\intv_{\so\dotsc\sigma_N}+\rho$. Together, the two bounds imply the claim. 
\end{proof}
\begin{defn}
\label{def:g_i_zero_and_g_i_one}Using $\ei$ and $\di$ from \cref{def:thetas} and $\pio$ and $\piz$ from \cref{def:phi_and_Phi}, we now define, for $i\in \mathbb{N}$, the following family
of functions:
\[
\gio:[0,1]\backslash\pio(0,1)\rightarrow\reals, \text{ where }
\gio(x)=\begin{cases}
-\frac{1-\ei}{\hf+\di}\cdot x+1 & \text{if }x\in[0,\hf+\di]\\
\frac{1-\ei}{\hf-2\di}\cdot x+\frac{-\hf-2\di+\ei}{\hf-2\di} & \text{if }x\in[\hf+2\di,1],
\end{cases}
\]
and $g_{0}^{(i)}(x)=g_{1}^{(i)}(1-x)$ with the following explicit closed-form expression:  
\begin{align*}
\giz:[0,1]\backslash\piz(0,1)\rightarrow\reals, \text{ where }  
\giz(x)&=\begin{cases}
-\frac{1-\ei}{\hf-2\di}\cdot x+1 & \text{if }x\in[0,\hf-2\di]\\
\frac{1-\ei}{\hf+\di}\cdot x+\frac{-\hf+\di+\ei}{\hf+\di} & \text{if }x\in[\hf-\di,1].
\end{cases}
\end{align*}
\end{defn}

\begin{lemma}\label{lem:g_properties}
The functions $\gisi$ from \cref{def:g_i_zero_and_g_i_one} satisfy the following properties. 
\begin{lemmenum}
\compresslist{
\item\label[lemma]{eq:g_one_g_zero_eval_one_zero} For both possible choices of $\sigma_i$,  the end points of the functions $\gisi$ satisfy  \[g_{\sigma_i}^{(i)}(1)=g_{\sigma_i}^{(i)}(0)=1.\] 
\item\label[lemma]{eq:g_one_g_one_mid_points}  For both possible choices of $\sigma_i$,  the functions $\gisi$ and $\pisi$ (from \cref{def:phi_and_Phi}) satisfy \[\gisi\circ \pisi(0) = \ei.\] 
\item\label[lemma]{all_slopes_of_gisi} Recall $\thib$ and $\thipb$ as defined in \cref{def:thetas}. The derivatives of $\giz$ and $\gio$ are given by \[ (\giz)^\prime(x) = \begin{cases}
  -\cot(\thib)  & \text{ if } x \in (0, \hf - 2\di)\\ 
  \cot(\thipb) & \text{ if } x\in (\hf - \di, 1)
\end{cases} \] 
and \[ (\gio)^\prime(x)  = \begin{cases}
    -\cot(\thipb) & \text{ if } x\in (0, \hf + \di)\\ 
    \cot(\thib) & \text{ if } x\in (\hf + 2\di, 1). 
\end{cases}  \]
\item\label[lemma]{bounds_on_slope} The derivative, in absolute value, of $\gisi$ for all $i \in \mathbb{N}$ is at least $\tfrac{1}{8}$ and at most $1$ (in the interior of its domain). In other words, we have: \[ \tfrac{1}{8}\leq  |(\giz)^\prime(x)| \leq 1\text{ for } x\in (0, \hf - 2\di)\cup(\hf-\di,1) \]and \[ \tfrac{1}{8}\leq  |(\gio)^\prime(x)| \leq 1\text{ for } x\in (0, \hf + \di)\cup(\hf+2\di,1). \]
}
\end{lemmenum}
\end{lemma}
\begin{proof}
    To check \cref{eq:g_one_g_zero_eval_one_zero} and \cref{eq:g_one_g_one_mid_points}, one may evaluate the functions in question from \cref{def:g_i_zero_and_g_i_one}. To prove \cref{all_slopes_of_gisi}, we observe from \cref{def:g_i_zero_and_g_i_one} that for $x\in (0, \hf - 2\di)$, the derivative of $\giz$ is given by the expression:  \[ (\giz)^\prime(x) = - \frac{1-\ei}{\hf-2\di} = - \frac{\frac{3}{2} \cdot \left( \frac{1}{2 \tan(\thib + \this) + \tan(\thib)} \right)}{\frac{1}{2}- 2\cdot\frac{1}{2} \cdot \left(\frac{\tan(\thib + \this) - \tan(\thib)}{2\cdot\tan(\thib + \this) + \tan(\thib)}\right)}  =  \frac{-3}{3\tan(\thib)} = -\cot(\thib),\]  and  for $x\in (\hf - \di, 1)$, the derivative of $\giz$ is: \[(\giz)^\prime(x) =  \frac{1-\ei}{\hf + \di} = \frac{\frac{3}{2} \cdot \left( \frac{1}{2 \tan(\thib + \this) + \tan(\thib)} \right)}{\frac{1}{2}+ \frac{1}{2} \cdot \left(\frac{\tan(\thib + \this) - \tan(\thib)}{2\cdot\tan(\thib + \this) + \tan(\thib)}\right)} = \frac{3}{3 \tan(\thib + \this)} = \cot(\thipb).\] The derivatives of $\gio$ are computed in a similar fashion. To prove \cref{bounds_on_slope}, we note that the largest angle $\thib$ obtained in the sequence described by \cref{eq:evolution_of_theta} is  \[\theta^{(\infty)}_{\mathrm{base}} = \theta^{(1)}_{\mathrm{base}} + \sum_{i = 1}^{\infty} \frac{\arctan(8)-\theta^{(1)}_{\mathrm{base}}}{2^i} \leq \theta^{(1)}_{\mathrm{base}}+ (\arctan(8)-\theta^{(1)}_{\mathrm{base}}) = \arctan(8).\numberthis\label{eq:largest_base_angle}\] Since the cotangent function is decreasing on the positive interval upto $\pi/2$, plugging into \cref{all_slopes_of_gisi} the upper bound \cref{eq:largest_base_angle} implies that the lower bound on the (absolute) value of the  derivative of $\gisi$ in the interior of its domain is \[ |(\gisi)^\prime(x)|\geq  \cot(\arctan(8)) = \tfrac{1}{8}.\] For the upper bound, we observe that the largest (in absolute value) derivative is attained at the smallest angle: \[ |(\gisi)^\prime(x)|\leq \cot(\tan(\theta^{(1)}_{\mathrm{base}})) = \cot(\tan(1)) = 1.\]  
\end{proof}

\begin{defn}\label{def:new_def_Phi}
    Using the above functions, we define 
\[\Phi^{(i)}(x)=\di\cdot x+\gisi\circ\pisi(0)-\di.\]
This definition yields the following important consequence: 
\begin{equation}
\Phi^{(i)}(1)= \giz\circ\piz(0) = \gio\circ\pio(0),\label{eq:Phi_one_g_circ_phi_zero_new}
\end{equation} where the second equality is justified in \cref{eq:g_one_g_one_mid_points}. 
\end{defn}

\noindent Finally we are ready to provide the following function definition. 
\begin{defn}
\label{def:Lipsc_cvx_fn_construction} 
Given $N\in \mathbb{N}$ and $\sigma\in\left\{0, 1\right\}^N$, define the function 
\[
\rNs(x)=\begin{cases}
\goso(x) & x\in\left[0,\quad\intv_{\sigma_{1}}[\ell]\right]\\
\Po\circ\gtst\circ\left(\poso\right)^{-1}(x) & x\in\left[\intv_{\sigma_{1}}[\ell],\quad\intv_{\sigma_{1}\st}[\ell]\right]\\
\Po\circ\Pt\circ\gthsth\circ\left(\poso\circ\ptst\right)^{-1}(x) & x\in\left[\intv_{\sigma_{1}\st}[\ell],\quad\intv_{\so\st\sigma_{3}}[\ell]\right],\\
\vdots & \vdots\\
\Po\circ\dotsc\circ\Phi^{(i)}\circ g_{\sigma_{i+1}}^{(i+1)}\circ\left(\poso\circ\dotsc\circ\pisi\right)^{-1}(x) & x\in\left[\intv_{\so\dotsc\si}[\ell],\quad\intv_{\so\dotsc\sigma_{i+1}}[\ell]\right]\\
\vdots & \vdots\\
\cot(\thiNpb)\cdot\left(\poso\circ\dotsc\circ\pNsN(0)-x\right)+\Po\circ\dotsc\circ\PN(1) & x\in\left[\intv_{\so\dotsc\sigma_{N}}[\ell],\quad x_{\mathrm{mid}}\right]\\
\cot(\thiNpb)\cdot\left(x-\poso\circ\dotsc\circ\pNsN(1)\right)+\Po\circ\dotsc\circ\PN(1) & x\in\left[x_{\mathrm{mid}},\quad\intv_{\so\dotsc\sigma_{N}}[r]\right]\\
\vdots\\
\Po\circ\dotsc\circ\Phi^{(i)}\circ g_{\sigma_{i+1}}^{(i+1)}\circ\left(\poso\circ\dotsc\circ\pisi\right)^{-1}(x) & x\in\left[\intv_{\so\dotsc\sigma_{i+1}}[r],\quad\intv_{\so\dotsc\si}[r]\right]\\
\vdots\\
\Po\circ\Pt\circ\gthsth\circ\left(\poso\circ\ptst\right)^{-1}(x) & x\in\left[\intv_{\sigma_{1}\st\sigma_{3}}[r],\quad\intv_{\so\st}[r]\right],\\
\Po\circ\gtst\circ\left(\poso\right)^{-1}(x) & x\in\left[\intv_{\sigma_{1}\st}[r],\quad\intv_{\so}[r]\right]\\
\goso(x) & x\in\left[\intv_{\so}[r],\quad1\right]\\
1-x & x<0\\
x & x>1,
\end{cases} \numberthis\label{eq:construction_of_rN}
\]
where $x_{\mathrm{mid}}= \hf\cdot\left(\poso\circ\ptst\circ\dotsc\circ\pNsN(0)+\poso\circ\ptst\circ\dotsc\circ\pNsN(1)\right)$.
\end{defn}
\noindent 
The following lemma follows immediately by combining \Cref{unique_chains_nonintersecting_intervals} with
the construction of $\rNs$, and noting that there

\begin{lemma}\label{lem:r_cantGuess}
For any $k<N\in \mathbb{N},~\sigma\in\left\{0, 1\right\}^N$, any local oracle $\OO$ and any $x\notin \intv_{\so\dotsc\sigma_k}$, it holds that $\OO_{\rNs}(x)$ does not depend on $\sigma_{k+1},\dots,\sigma_N$.
Therefore, for any $t\in\NN,~1\leq k<l< N:$
\[
    \Pr_{\sigma\sim\mathrm{Unif}\{0,1\}^N}[x_{t+1}\in \intv_{\so\dotsc\sigma_k\dotsc\sigma_l}\mid x_1,\dots,x_t\notin \intv_{\so\dotsc\sigma_k}]\leq \frac{1}{2^{l-k-1}}~.\]
\end{lemma}

We now prove some properties of $\rNs$ and construct the function $\bar{h}$ referred to in   \Cref{lem: basic hard}; our construction of $\bar{h}$ ensures that it inherits all its necessary properties from $\rNs$. 
\subsection{Putting it all together: proof of \cref{lem: basic hard}}\label{sec:proof of lem basic hard}
\begin{proof}[Proof of \cref{lem: basic hard}] We first show 
the properties of continuity, $1$-Lipschitzness, and convexity in $\rNs$, and its upper bound at zero (i.e., \Cref{lem: basic hard i}).     
\paragraph{Proof of continuity.}
We start by proving the continuity of $\rNs.$ Since $g_{\sigma_{i}}^{(i)}$,
$\phi_{\sigma_{i}}^{(i)},$and $\Phi^{(i)}$ are all (piecewise) affine
over their domains of definition, their composition and hence, from
\cref{def:Lipsc_cvx_fn_construction}, $\rNs$ is also piecewise affine.
To show continuity, we therefore need to show this only at the endpoints
of each of the segments in \cref{def:Lipsc_cvx_fn_construction}. For
some $1<i<N,$ consider the left endpoint of the segment \[\left[\intv_{\so\st\dotsc\si}[\ell],\quad\intv_{\so\st\dotsc\sigma_{i+1}}[\ell]\right].\] 
For continuity at this interval's left endpoint $\intv_{\so\st\dotsc\si}[\ell] = \poso\circ\ptst\circ\dotsc\circ\pisi(0)$ (\cref{eq:interval_in_terms_of_phi_chain_at_zero}),  we need to show that at $x=\poso\circ\ptst\circ\dotsc\circ\pisi(0),$
 the functions $\Po\circ\Pt\circ\dotsc\circ\Phi^{(i-1)}\circ g_{\sigma_{i}}^{(i)}\circ\left(\poso\circ\ptst\circ\dotsc\circ\phi_{\sigma_{i-1}}^{(i-1)}\right)^{-1}(x)$
and $\Po\circ\Pt\circ\dotsc\circ\Phi^{(i)}\circ g_{\sigma_{i+1}}^{(i+1)}\circ\left(\poso\circ\ptst\circ\dotsc\circ\pisi\right)^{-1}(x)$
have the same value. To prove this, we simply evaluate the two functions
one by one. First, observe that 
\begin{align*}
 & \Po\circ\Pt\circ\dotsc\circ\Phi^{(i)}\circ g_{\sigma_{i+1}}^{(i+1)}\circ\left(\poso\circ\ptst\circ\dotsc\circ\pisi\right)^{-1}\left(\poso\circ\ptst\circ\dotsc\circ\pisi(0)\right)\\
 & =\Po\circ\Pt\circ\dotsc\circ\Phi^{(i)}\circ g_{\sigma_{i+1}}^{(i+1)}(0)=\Po\circ\Pt\circ\dotsc\circ\Phi^{(i)}(1),\numberthis\label{eq:first-lipsc-and-cvx}
\end{align*}
where we use \cref{eq:g_one_g_zero_eval_one_zero} in the final step.
Next, observe that 
\begin{align*}
 & \Po\circ\Pt\circ\dotsc\circ\Phi^{(i-1)}\circ g_{\sigma_{i}}^{(i)}\circ\left(\poso\circ\ptst\circ\dotsc\circ\phi_{\sigma_{i-1}}^{(i-1)}\right)^{-1}\left(\poso\circ\ptst\circ\dotsc\circ\pisi(0)\right)\\
 & =\Po\circ\Pt\circ\dotsc\circ\Phi^{(i-1)}\circ g_{\sigma_{i}}^{(i)}\circ\phi_{\sigma_{i}}^{(i)}(0)
 .\numberthis\label{eq:second-lipsc-and-cvx}
\end{align*}
From \cref{eq:Phi_one_g_circ_phi_zero_new}, the final terms in \cref{eq:first-lipsc-and-cvx}
and \cref{eq:second-lipsc-and-cvx} may be concluded to be equal. Next,
consider $i=1.$ We need to consider the left endpoint of the
interval \[ \left[\intv_{\sigma_{1}}[\ell],\quad\intv_{\sigma_{1}\st}[\ell]\right]. \]
To check continuity at the left endpoint of this interval, we evaluate two functions $\goso(x)$ and $\Po\circ\gtst\circ\left(\poso\right)^{-1}(x)$
at $\intv_{\sigma_{1}}[\ell] = \poso(0).$ We have: 
\begin{equation}
\goso(x)=g_{\sigma_{1}}^{(1)}\circ\poso(0)
.\label{eq:third-lipsc-and-cvx}
\end{equation}
Next, we
have 
\begin{equation}
\Po\circ\gtst\circ\left(\poso\right)^{-1}(x)=\Po\circ\gtst\circ\left(\poso\right)^{-1}(\poso(0))=\Po\circ g_{\sigma_{2}}^{(2)}(0)=\Po(1)
.\label{eq:fourth-lipsc-and-cvx}
\end{equation}
Comparing \cref{eq:third-lipsc-and-cvx}
and \cref{eq:fourth-lipsc-and-cvx} using \cref{eq:Phi_one_g_circ_phi_zero_new} completes the proof for $i=1.$
Next, we consider the left endpoint of the ``left middle'' segment 
\[ 
\left[\intv_{\so\st\dotsc\sigma_{N}}[\ell],\quad x_{\mathrm{mid}}\right].
\]
To evaluate continuity at the left endpoint of this segment, we need to evaluate at $\intv_{\so\st\dotsc\sigma_{N}}[\ell] = \poso\circ\ptst\circ\dotsc\circ\pNsN(0)$
the functions $\Po\circ\Pt\circ\dotsc\circ\Phi^{(N-1)}\circ g_{\sigma_{N}}^{(N)}\circ\left(\poso\circ\ptst\circ\dotsc\circ\phi_{\sigma_{N-1}}^{(N-1)}\right)^{-1}(x)$ and $\cot(\thiNpb)\cdot\left(\poso\circ\ptst\circ\dotsc\circ\pNsN(0)-x\right)+\Po\circ\Pt\circ\dotsc\circ\PN(1)$. 
To this end, first observe 
that 
\begin{align*}
&\Po\circ\Pt\circ\dotsc\circ\Phi^{(N-1)}\circ g_{\sigma_{N}}^{(N)}\circ\left(\poso\circ\ptst\circ\dotsc\circ\phi_{\sigma_{N-1}}^{(N-1)}\right)^{-1}(\poso\circ\ptst\circ\dotsc\circ\pNsN(0))\\
&=\Po\circ\Pt\circ\dotsc\circ\Phi^{(N-1)}\circ g_{\sigma_{N}}^{(N)}\circ\phi_{\sigma_{N}}^{(N)}(0)
.\numberthis\label{eq:sixth-lipsc-and-cvx}
\end{align*}
Next, observe that evaluating the other function at $x = \poso\circ\ptst\circ\dotsc\circ\pNsN(0)$ gives: 
\begin{align*}
&\cot(\thiNpb)\cdot\left(\poso\circ\ptst\circ\dotsc\circ\pNsN(0)-\poso\circ\ptst\circ\dotsc\circ\pNsN(0)\right)+\Po\circ\Pt\circ\dotsc\circ\PN(1)\\
&= \Po\circ\Pt\circ\dotsc\PN(1)\\ &=\Po\circ\Pt\circ\dotsc\circ\Phi^{(N-1)}\circ\Phi^{(N)}(1)
.\numberthis\label{eq:fifth-lipsc-and-cvx}
\end{align*}
We may again use \cref{eq:Phi_one_g_circ_phi_zero_new} to equate \cref{eq:sixth-lipsc-and-cvx} and \cref{eq:fifth-lipsc-and-cvx}. Finally, we show continuity at the ``midpoint'' \[x_{\mathrm{mid}}= \hf\cdot\left(\poso\circ\ptst\circ\dotsc\circ\pNsN(0)+\poso\circ\ptst\circ\dotsc\circ\pNsN(1)\right).\] From \Cref{def:Lipsc_cvx_fn_construction}, we note that the left of the two middle segments (i.e., $x\in\left[\intv_{\so\st\dotsc\sigma_{N}}[\ell],\quad x_{\mathrm{mid}}\right]$) is \[ \rNs(x):=\cot(\thiNpb)\cdot\left(\poso\circ\ptst\circ\dotsc\circ\pNsN(0)-x\right)+\Po\circ\Pt\circ\dotsc\circ\PN(1).\]   Similarly,  we have 
that the segment to the right of $x_{\mathrm{mid}}$ (i.e., $x\in\left[x_{\mathrm{mid}},\quad\intv_{\so\st\dotsc\sigma_{N}}[r]\right]$) is described by 
\[ \rNs(x):= \cot(\thiNpb)\cdot\left(x-\poso\circ\ptst\circ\dotsc\circ\pNsN(1)\right)+\Po\circ\Pt\circ\dotsc\circ\PN(1).\] From these two definitions, we can check that the values of $\rNs(x_{\mathrm{mid}})$ from both the definitions coincide. The continuity at the  endpoints of segments to the right of $x_{\mathrm{mid}}$ may be similarly established. 

\paragraph{Proof of $1$-Lipschitzness.} Since $\rNs$ is piecewise linear and continuous, the proof of its Lipschitzness requires only proving Lipschitzness of each of the segments it is composed of. We use the following two observations to establish this fact. First, \[ (\Phi^{(i)})^\prime = (\phi^{(i)})^\prime, \]
 because $\Phi^{(i)}$ and $\phi^{(i)}$ differ by only a constant additive factor. Second, by application of the chain rule, one may note that  the derivative of a composition of affine functions equals the product of the derivatives of the individual composing functions. Using these two facts, we observe that \[ \lvert\left(\Po\circ\dotsc\Phi^{(i-1)} \circ \gisi \circ (\poso \circ \dotsc \phi^{(i)}_{\sigma_{i-1}})^{-1}\right)^{\prime}(x)\rvert=  |(\gisi)^\prime(x)| \leq  1,\] where we used the upper bound from \cref{bounds_on_slope} in the final step. For the two middle segments, the absolute value of the  slope is $\cot(\thiNpb)$. From the proof of \Cref{bounds_on_slope} and the fact that the cotangent function is non-increasing in $(0, \pi/2)$, we have $\cot(\thiNpb)\leq \cot(\theta^{(1)}_{\mathrm{base}}) = \cot(\arctan(1)) = 1$. 
This concludes the proof of $\rNs$ being a $1$-Lipschitz function. 

 \paragraph{Proof of convexity.} To prove convexity of $\rNs$, we 
  show that the derivatives of consecutive segments composing $\rNs$ are non-decreasing. Consider the segment $\left[\intv_{\so\st\dotsc\sigma_{i-1}}[\ell],\quad\intv_{\so\st\dotsc\sigma_{i}}[\ell]\right]$. We have that for some $x\in \left[\intv_{\so\st\dotsc\sigma_{i-1}}[\ell],\quad\intv_{\so\st\dotsc\sigma_{i}}[\ell]\right]$, the derivative of $\rNs$ at this point $x$ is: 
\begin{align*}
    \left( \Po\circ \dotsc \Phi^{(i-1)} \circ g^{(i)}_{\sigma_{i}} \circ (\poso \circ\ptst \circ\dotsc \phi^{(i-1)}_{\sigma_{i-1}})^{-1} \right)^\prime(x) =  \left(g^{(i)}_{\sigma_{i}}\right)^\prime(y) \leq  - \cot(\thipb),
\end{align*} where $y\in [0, \pisi(0)]$ and for the final bound, we used from \cref{all_slopes_of_gisi} the largest possible value for  the slope of the left segment of $g^{(i)}_{\sigma_{i}}$ (for $\gisi$, the segment $[0, \pisi]$ corresponds to the left segments of $\gisi$). 
In a similar fashion,  the smallest derivative of $\rNs$ on some $x$ in the segment $\left[\intv_{\so\st\dotsc\sigma_{i}},\quad\intv_{\so\st\dotsc\sigma_{i+1}}\right]$ may be derived to be at least \[  \left( \Po\circ \dotsc \Phi^{(i)} \circ g^{(i+1)}_{\sigma_{i+1}} \circ (\poso \circ\ptst \circ\dotsc \phi^{(i)}_{\sigma_{i}})^{-1} \right)^\prime(x) 
\geq -\cot(\thipb) 
. \] Therefore, the derivatives of consecutive segments composing $\rNs$ to the left  of the middle segments are non-decreasing, going from left to right. The proof analogously extends to the right  of the middle segments. For the two middle segments, note that the left segment has slope $-\cot(\thiNpb)$ and the right segment has slope $\cot(\thiNpb)$; since $\thiNpb\in (\arctan(1), \arctan(8)) \subseteq (0, \pi/2)$, this means the slope is increasing when crossing $x_{\mathrm{mid}}$ from left to right. Finally, going from the left of the middle left segment to the middle left segment, the slope can only increase because of our choice of slope of the middle left segment; the analogous argument applies to the right side. Thus, overall, we have shown that the slope  increases going from left to right, thus proving the convexity of $\rNs$.

\paragraph{Proof of $\rNs(0) \leq 2$.} From \Cref{def:Lipsc_cvx_fn_construction}, one may check that $\rNs(0) = \goso(0)$, which from  \Cref{eq:g_one_g_zero_eval_one_zero} satisfies $\goso(0)=1$, thus proving our claim.   

\vspace{1em}

We now prove \Cref{lem: basic hard ii}, i.e., that $\rNs$ has a unique minimizer, that the absolute value of its slope is lower bounded by a constant, and obtain the expression for its minimum value. 

 \paragraph{Proof of lower bound on absolute slope.} For any $i\in [2, N]$, consider a point $x$ that lies in the segment $\left(\intv_{\so\st\dotsc\sigma_{i-1}}[\ell],\quad\intv_{\so\st\dotsc\sigma_{i}}[\ell]\right)$. We have that the derivative of $\rNs$ at this point is given by: 
\begin{align*}
    \lvert\left( \Po\circ \dotsc \Phi^{(i-1)} \circ g^{(i)}_{\sigma_{i}} \circ (\poso \circ\ptst \circ\dotsc \phi^{(i-1)}_{\sigma_{i-1}})^{-1} \right)^\prime(x)\rvert =  \lvert\left(g^{(i)}_{\sigma_{i}}\right)^\prime(y)\rvert\geq \tfrac{1}{8}, 
\end{align*} where $y\in (0, \pisi(0))$ and the final lower bound follows by applying \cref{bounds_on_slope}. The same lower bound may be similarly obtained for segments to the right of the middle segment. For the two middle segments, the absolute value of the slope is $\cot(\thiNpb)$. From the non-increasing property of the cotangent function and our initial choice of $\theta^{(1)}_{\mathrm{base}}$, we have $\cot(\thiNpb)\geq \cot(\arctan(8))=1/8$, thus concluding the proof, overall, of the lower bound on the absolute value of the slope. 

\paragraph{Proof of unique minimizer of $\rNs$.} As we just showed,  $\rNs$ is convex. At the point $x_{\mathrm{mid}}$, the slope of $\rNs$ changes from $-\cot(\thiNpb)$ to $\cot(\thiNpb)$, which implies that there exists a zero subgradient of $\rNs$ at $x_{\mathrm{mid}}$. Hence, $x_{\mathrm{mid}}$ is \emph{a} minimizer of $\rNs$. This minimizer is unique because the two segments of the function intersecting at it both have non-zero slopes. This minimum value  is obtained by evaluating $\rNs$ at $x_{\mathrm{mid}}$: \[ \cot(\thiNpb)\cdot\hf\left(\poso\circ\ptst\circ\dotsc\circ\pNsN(0)-\poso\circ\ptst\circ\dotsc\circ\pNsN(1)\right)+\Po\circ\Pt\circ\dotsc\circ\PN(1).\] 

\paragraph{Construction of $\hbNs$.}
To use $\rNs$ to construct $\hbNs$  as defined in \Cref{lem: basic hard}, we define \[\hbNs = \rNs+2-\rNs(x_{\mathrm{mid}}). 
\] Since we only added a constant term to $\rNs$, its properties of continuity, $1$-Lipschitzness, convexity, unique minimizer $x^*= x_{\mathrm{mid}}$ in $(0,1)$, and lower bound of $1/8$ on slope remain preserved. 
Next, note that since $\rNs\geq \rNs(x_{\mathrm{mid}})$, it implies $\hbNs\geq 2$, which proves the assertion that $\hbNs:\mathbb{R}\mapsto[2,\infty)$. 
Finally, at $x=0$, we have $\hbNs(0)=2 + \rNs(0)-\rNs(x_{\mathrm{mid}})\leq 3$ because $\rNs$ is a piecewise affine function over $[0,1]$ with maximum slope being $1$, and so its maximum range of values can be $1$. Thus, our constructed $\bar{h}$ satisfies \Cref{lem: basic hard i} and \Cref{lem: basic hard ii}.

\paragraph{Proof of \Cref{lem: basic hard iii}}

Let $k=\tfrac{1}{4}\sqrt{\log\left(\tfrac{1}{\rho}\right)}$ 
 for some $\rho>0$ to be determined, and $N=k+1,~\sigma\sim\mathrm{Unif}\{0,1\}^N$. Consider the iterates of $\Acal$, $x_1,\dots,x_T$,
as random variables when applied to the (random) function $\hbNs$. 
Since $x^*=x_{\mathrm{mid}}\in \intv_{\so\dotsc\sigma_N}$, we have
\begin{align*}
\Pr[\exists t\in[T]:|x_t-x^*|\leq\rho]
\leq
\Pr[\exists t\in[T]:\dist(x_t,\intv_{\so\dotsc\sigma_N})\leq\rho]
\leq
\Pr[\exists t\in[T]:x_t\in\intv_{\so\dotsc\sigma_{k}}],
\end{align*}
where the second inequality follows from  \Cref{dist_inv_N_less_del_implies_x_in_intv_k}. By denoting the ``progress tracking'' stochastic process
\[
Z_0:=0,~Z_t:=\max\{l\in\NN:\exists s\leq t, x_s\in\intv_{\so\dotsc\sigma_l}\}~,
\]
we note that $Z_{t+1}-Z_{t}\geq0$ with probability 1, and moreover that $\Pr[Z_{t+1}-Z_{t}=m]\leq{2^{-(m-1)}}$ by \Cref{lem:r_cantGuess}. Hence
\begin{align*}
\Pr[\exists t\in[T]:x_t\in\intv_{\so\dotsc\sigma_{k}}]
&=\Pr[Z_T\geq k]
\leq \frac{1}{k}\E[Z_T]
=\frac{1}{k}\sum_{j=1}^{T}\E[Z_j-Z_{j-1}]
\\&=\frac{1}{k}\sum_{j=1}^{T}\sum_{m=0}^{\infty}\Pr[Z_j-Z_{j-1}=m]
\leq\frac{1}{k}\sum_{j=1}^{T}\sum_{m=0}^{\infty}2^{-(m-1)}
\\
&=\frac{4T}{k}= \frac{16T}{\sqrt{\log(1/\rho)}}~.
\end{align*}
Finally, by setting $\rho:=\exp(-256T^2/\gamma^2)$
the quantity above is bounded by $\gamma$, completing the proof.

 \end{proof}

\section{Proof of \Cref{lem: choice of w}}\label{sec:proof of lem with many cases}

\noindent The goal of this section is to prove \Cref{lem: choice of w}, which we first recall below. 
\lemmaWithManyCases*

\begin{proof}[Proof of \cref{lem: choice of w}] 
Throughout the proof we omit the subscripts $\bw,\mu$. We recall  that $
f(\bx):=\max\left\{h(\bx)-\sigma_{\mu}( q(\bx-\bx^*)),0\right\}$ by \cref{def:fwu}. 

~\paragraph{Proof of \Cref{lem: choice of w i}.} 
It is clear that $f$ is non-negative by design. 
The function $h$ (in \cref{def:fwu}) is $\tfrac{1}{2}$-Lipschitz (by design in \cref{lem: d>=2 hard}); the function $\sigma_\mu$ (in \cref{def:fwu}) has the following derivative: 
\[\sigma_\mu^\prime(z) = \begin{cases}
    0, & z\leq 0 \\
    \tfrac{z}{4\mu}, & z\in(0,\mu] \\
    \tfrac{1}{4}, & z>\mu
\end{cases} \numberthis\label{eq:derSigmaMu}\] and is therefore $\tfrac{1}{4}$-Lipschitz. Further, $\bz\mapsto\norm{\bz},~\bz\mapsto\inner{\bar{\bw},\bz},~\bz\mapsto\max\{\bz,0\}$ and translations are all $1$-Lipschitz. Finally, 
the summation (respectively, positive scaling) of functions results in a summation (respectively, scaling) of Lipschitz constants; and
the composition of functions yields a product of Lipschitz constants. These imply that $q$ (as in \cref{def:q_and_sigma_mu}) is $\tfrac{3}{2}$-Lipschitz and  $f$ is 
$1$-Lipschitz. 

To prove Clarke regularity of $f$, we start by examining the function \[\phi(\bx):=-\sigma_\mu(q(\bx)),\numberthis\label{def:phi_in_terms_of_sigma_mu_q}\] and note that $\phi$ is continuously differentiable over $\reals^d\setminus\{-\bw\}$ with
\[
\nabla \phi(\bx)
=-\sigma_{\mu}'(\inner{\bar{\bw},\bx+\bw}-\tfrac{1}{2}\norm{\bx+\bw})\cdot (\bar{\bw}-\tfrac{1}{2}(\overline{\bx+\bw})),
~\forall \bx\neq -\bw.
\]
In particular, by \eqref{eq:derSigmaMu}, we see that
$\lim_{\bx\to-\bw}\nabla\phi(\bx)=\bzero$. Further note that
\begin{align*}
{\lim\sup}_{\bv\to\bzero}\frac{|\phi(-\bw+\bv)-\phi(-\bw)|}{\norm{\bv}}
&={\lim\sup}_{\bv\to\bzero}\frac{|\phi(-\bw+\bv)|}{\norm{\bv}}
={\lim\sup}_{\bv\to\bzero}\frac{|\sigma_{\mu}(\inner{\bar{\bw},\bv}-\tfrac{1}{2}\norm{\bv})|}{\norm{\bv}}=0,
\end{align*}
where the last equality follows from
$\inner{\bar{\bw},\bv}-\tfrac{1}{2}\norm{\bv}\overset{\bv\to\bzero}{\longrightarrow}\bzero$
and
$\lim_{z\to0}\sigma'_{\mu}(z)=0$.
Therefore $\phi$ is differentiable at $-\bw$ with $\nabla\phi(-\bw)=\bzero$, hence everywhere continuously differentiable, which implies by \cref{item:prop_cont_diff_implies_reg} that for $\phi$ as defined in \cref{def:phi_in_terms_of_sigma_mu_q}, it holds that  \[\phi \text{ is regular}.\numberthis\label{eq:phi_is_regular}
\] 
Since $\phi$ is regular, the shifted function $\phi(\cdot-\bx^*)$ is regular as well. Moreover,  since $h$ is convex and Lipschitz, it is regular as well (\Cref{item:prop_reg_convex}), hence by \cref{item:prop_reg_finsum} we have that \[h(\cdot)+\phi(\cdot-\bx^*) \text{ is regular}.
\] Finally, since $\max\{\cdot,0\}$ is convex and Lipschitz, we conclude by 
\cref{{item:prop_reg_convex}}
that it is regular; hence, by \cref{fact:regularity_of_composition}, the composition  $f_{\bw,\mu}(\bx)=\max\left\{h(\bx)-\sigma_{\mu}( q(\bx-\bx^*)),0\right\}$ is regular. 

~\paragraph{Proof of \Cref{lem: choice of w ii}.} Our proof for this part of the lemma closely follows that by \citet{tian2021hardness} and \citet{kornowski2022oracle}. 
We start by proving that, for $c< \tfrac{1}{100}$, the function $\varphi(\bx):=h(\bx+\bx^*)-\sigma_{\mu}\left(\inner{\bar{\bw},\bx+\bw}-\tfrac{1}{2}\norm{\bx+\bw}\right)$ has no $c$-stationary points, by an exhaustive case analysis showing a universal positive lower bound (of $\tfrac{1}{100}$) on the absolute value of the minimum norm element of its Clarke subdifferential in every case.
\begin{itemizec}
    \item $\bx=\bzero:$ By applying \cref{fact:prop_clarke_subdifferential_calculus}, the fact that $\|\mathbf{w}\|\gg \mu$, and \cref{{eq:derSigmaMu}}, it holds that $$\partial \varphi(\bzero) = \left\{\partial h(\bx^*) - \partial \sigma_{\mu}(\tfrac{1}{2}\|\bw\|) \right\}
    =\left\{\partial h(\bx^*)-\tfrac{1}{8}\bar\bw\right\}.$$
    By further examining the definition of $h$ (from \cref{{lemma_item_prob_premain_proof}}), we may simplify this to 
    \[
    \partial \varphi(\bzero)
    =\left\{\partial h(\bx^*)-\tfrac{1}{8}\bar\bw\right\} 
    \subseteq \left\{ \tfrac{1}{32}\bu+\lambda\be_d-\tfrac{1}{8}\bar\bw~|~\lambda\in\partial\bar{h}(x^*),\norm{\bu}\leq1,\bu\perp\be_d\right\}~,
    \] where $\bar{h}$ is as defined in \cref{{lem: d>=2 hard}}. 
    Since projecting any  vector from the above set 
    $\partial \varphi(\bzero)$ onto $\mathrm{span}(\be_d)^\perp$ cannot increase its norm, we may conclude that
    \[
    \norm{\tfrac{1}{32}\bu +  \lambda\be_d -\tfrac{1}{8}\bar\bw}
    \geq \norm{\tfrac{1}{32}\bu-\tfrac{1}{8}\bar\bw} 
    \geq\tfrac{1}{8}-\tfrac{1}{32}=\tfrac{3}{32}.
    \]

    \item $\bx=-\bw:$ Recall in \cref{{eq:phi_is_regular}} we proved the Clarke regularity of $\phi(\bx):=-\sigma_\mu(\inner{\bar{\bw},\bx+\bw}-\tfrac{1}{2}\norm{\bx+\bw})$. Applying this property at $-\bw$ and noting, from \cref{{eq:derSigmaMu}}, that $\partial \phi(-\bw)=\bzero$, we get that
    \[
    \partial\varphi(-\bw)
    =\left\{\partial h(-\bw+\bx^*)-\partial\phi(-\bw)\right\}
    \subset\left\{\tfrac{1}{32}\bar\bw+\lambda\be_d-\bzero~|~\lambda\in\partial\bar{h}(-w_d+x^*)\right\}, 
    \] where we used the fact that $\bx^*\perp \be_d$ (from \cref{{lemma_item_prob_premain_proof}}) and $\bw\perp\be_d$. 
  By projecting any such vector onto $\mathrm{span}(\be_d)^\perp$, we see that it clearly has norm of at least $\frac{1}{32}$. 

    \item $x_d\neq 0:$ It holds that 
    \[\partial\varphi(\bx)\subseteq \left\{\tfrac{1}{32}\bu+\lambda\be_d-s\left(\bar{\bw}-\tfrac{1}{2}\bv\right)~|~\lambda\in\partial\bar{h}(x_d+x^*),\bu\perp\be_d,~s\in[0,\tfrac{1}{4}],\norm{\bv}\leq1\right\}~,
    \] where we used \cref{eq:derSigmaMu} to evaluate $\partial \phi(\bx)$. Next, since $x_d\neq 0$, it implies that $x_d+x^*\neq x^*$,
which implies, by \cref{{item:unique_min_min_slope}},  that 
$\forall\lambda\in\partial\bar{h}(x_d+x^*):|\lambda|\geq \tfrac{1}{16}$. 
Then, projecting any vector from the set above onto $\be_d$, we see that
    \[
    \norm{\tfrac{1}{32}\bu + \lambda\be_d -s\left(\bar{\bw}-\tfrac{1}{2}\bv\right)}
    \geq | |\lambda|-\tfrac{1}{4}\cdot\tfrac{1}{2}| \geq \tfrac{1}{16}~.
    \]

    \item $x_d=0,~\bx\notin \{\bzero,-\bw\}, ~\inner{\bar{\bw},\overline{\bx+\bw}}<\tfrac{1}{2}:$ Note that for any such $\bx$, we have that $\inner{\bar{\bw},{\bx+\bw}}-\tfrac{1}{2}\norm{\bx+\bw} = \norm{\bx+\bw}\cdot(\inner{\bar{\bw},\overline{\bx+\bw}}-\tfrac{1}{2})<0$, which  implies that  $\sigma_\mu(\inner{\bar{\bw},\bx+\bw}-\tfrac{1}{2}\norm{\bx+\bw})=0$. This in turn  implies $\phi(\bx)=0$. Consequently, $\varphi(\bx)$ is locally identical to $h(\bx+\bx^*)=\bar{h}(x_d+x^*)+\tfrac{1}{32}\norm{\bx_{1:d-1}} = \bar{h}(x^*) + \tfrac{1}{32}\norm{\bx_{1:d-1}}$, where we used $\bx^*_{1:d-1}=\mathbf{0}_{d-1}$ (from \cref{{lemma_item_prob_premain_proof}}) in the first step and $x_d=0$ (as assumed in this case) in the final step. Since $x_d=0$ yet $\bx\neq \bzero$, it must be that $\bx_{1:d-1}\neq \bzero_{d-1}$, thus
    \[
    \partial\varphi(\bx) = \left\{ \partial \bar{h}(x^*) + \tfrac{1}{32} \partial \|\bx_{1:d-1}\| \right\}\subseteq \left\{\tfrac{1}{32}\overline{\bx_{1:d-1}} + \lambda \be_{d}~|~\lambda\in\partial h(x^*)\right\}.
    \]
    For any $\mathbf{g}\in  \partial\varphi(\bx)$, we have that $\|\mathbf{g}\|\geq \inner{\mathbf{g}, \overline{\bx_{1:d-1}}} \geq \tfrac{1}{32}$.
    
    \item $x_d=0,~\bx\notin \{\bzero,-\bw\}, ~\inner{\bar{\bw},\overline{\bx+\bw}}>\tfrac{1}{2}+\tfrac{\mu}{\norm{\bx+\bw}}:$ In this case, we have $\inner{\bar{\bw},\bx+\bw}-\tfrac{1}{2}\norm{\bx+\bw} = \left(\inner{\bar{\bw},\overline{\bx+\bw}}-\tfrac{1}{2}\right)\norm{\bx+\bw} \geq \mu$. Then, noting from \cref{{def:q_and_sigma_mu}} the definition of $\sigma_{\mu}$ for the appropriate range of the argument yields that  for any such $\bx:\sigma_\mu(\inner{\bar{\bw},\bx+\bw}-\tfrac{1}{2}\norm{\bx+\bw})=\tfrac{1}{4}\inner{\bar{\bw},\bx+\bw}-\tfrac{1}{8}\norm{\bx+\bw}-\tfrac{\mu}{8}$.  
    Combining this with $x_d=0$, we have that $\varphi(\bx)$ is locally identical to
    \[
    \bx\mapsto\bar{h}(x^*)+\tfrac{1}{32}\norm{\bx_{1:d-1}}-\tfrac{1}{4}\inner{\bar{\bw},\bx+\bw}+\tfrac{1}{8}\norm{\bx+\bw}+\tfrac{\mu}{8}~.
    \]
    As in the previous case, since $x_d=0$ yet $\bx\neq \bzero$, we have that $\bx_{1:d-1}\neq \bzero_{d-1}$, thus 
    \[
    \partial\varphi(\bx)\subseteq
    \left\{\tfrac{1}{32}\overline{\bx_{1:d-1}} + \lambda\be_d -\tfrac{1}{4}\bar\bw+\tfrac{1}{8}(\overline{\bx+\bw})~|~\lambda\in\partial h(x^*)
    \right\}~.
    \]
    Hence, for $\bg\in\partial\varphi(\bx)$, it holds that 
    \[
    \norm{\bg}\geq \inner{\bg,-\bar\bw}
    =-\frac{1}{32}\inner{\bar\bx,\bar\bw}+\frac{1}{4}-\frac{1}{8} \inner{\overline{\bx+\bw},\bar\bw}
    \geq \frac{1}{4}-\frac{5}{32}=\frac{3}{32}~.
    \]

    \item $x_d=0,~\bx\notin \{\bzero,-\bw\}, ~\|\bx+\bw\|\leq 10\mu, ~\inner{\bar{\bw},\overline{\bx+\bw}}\in[\tfrac{1}{2},\tfrac{1}{2}+\tfrac{\mu}{\norm{\bx+\bw}}]:$ 
    We have that \[\partial \varphi(\bx) = \left\{\tfrac{1}{32} \overline{\bx_{1:d-1}} + \lambda \be_d - v\cdot (\bar{\bw} 
 - \tfrac{1}{2}\overline{\bx+\bw})~|~\lambda\in\partial h(x^*)\right\},\numberthis\label{eq:last_case_one_diff_exp}\] where $0\leq v:= \tfrac{1}{4\mu}(\inner{\bar{\bw}, \bx+\bw} - \tfrac{1}{2}\|\bx+\bw\|)\leq 1$. The claimed bounds on $v$ follow from the assumption $\inner{\bar{\bw},\overline{\bx+\bw}}\in[\tfrac{1}{2},\tfrac{1}{2}+\tfrac{\mu}{\norm{\bx+\bw}}]$. The final term in \cref{eq:last_case_one_diff_exp} comes from plugging in the appropriate argument in \cref{eq:derSigmaMu} following the deduced range of $v$. We now proceed to show a lower bound on $-\bar{\bw}^\top \bar{\bx}$, which we will use to show the desired lower bound on \cref{eq:last_case_one_diff_exp}. To this end, define $\bx^\prime:= -\frac{\bw^\top \bx}{\|\bx\|^2}\cdot\bx$ (which is a valid operation because we assume $\bx\neq \bzero$). We can verify that: \[\inner{\bx^\prime, \bx^\prime+\bw} = \|\bx^\prime\|^2 + \bw^\top \bx^\prime = 0,\numberthis\label{eq:inner_xp_xp_plus_w_zero}\]where the last step follows from plugging in the definition of $\bx^\prime$. Next, by starting with the assumption that $10\mu\geq \|\bx+\bw\|$ and expressing $\bx+\bw$ as $\bx-\bx^\prime+\bx^\prime+\bw$, we have
 \[100\mu^2\geq \|\bx+\bw\|^2 = \|\bx^\prime + \bw\|^2 + \|\bx-\bx^\prime\|^2 + 2 \inner{\bx-\bx^\prime, \bx^\prime+\bw}.\numberthis\label[ineq]{eq:tau_squared_lb_first}\] Now, since $\inner{\bx^\prime, \bx^\prime+\bw} =0$ (by \cref{eq:inner_xp_xp_plus_w_zero}) and $\bx$ is proportional to $\bx^\prime$ (by design of $\bx^\prime$), we have $\inner{\bx, \bx^\prime+\bw} =0$ as well. Consequently, we have $\inner{\bx-\bx^\prime, \bx^\prime+\bw}=0$, which implies, in \cref{eq:tau_squared_lb_first}, that \[ 100\mu^2 \geq  \|\bx^\prime + \bw\|^2 + \|\bx-\bx^\prime\|^2. \numberthis\label[ineq]{eq:tau_squared_lb_second}\] By recalling that $\|\bw\|= 1000\mu \geq \|\bx+\bw\|$, we have that ${\bw}^\top {\bx}\leq0$.  Combining this with the definition of $\bx^\prime$, we obtain $\bar{\bw}^\top \bar{\bx}=\bar{\bw}^\top \bar{\bx}^\prime$, from which we conclude that \[-\bar{\bw}^\top \bar{\bx} = -\bar{\bw}^\top \bar{\bx}^\prime = \frac{\|\bx^\prime\|}{\|\bw\|}. \numberthis\label{eq:minus_barW_barXprime_equals_frac}\] Next, by again using $\inner{\bx^\prime, \bx^\prime+\bw}=0$, we have $\|\bw\|^2 = \|\bx^\prime\|^2 + \|\bx^\prime + \bw\|^2$. Then, we have   that  $ \|\bw\|^2 = \|\bx^\prime\|^2 + \|\bx^\prime + \bw\|^2 \leq \|\bx^\prime\|^2 + 100\mu^2$, where we used $\|\bx^\prime  + \bw\|^2\leq 100\mu^2 $ from \cref{eq:tau_squared_lb_second}. Rearranging and combining with \cref{eq:minus_barW_barXprime_equals_frac} yields \[ -\bar{\bw}^\top \bar{\bx} = \frac{\|\bx^\prime\|}{\|\bw\|} \geq \sqrt{1 - \frac{100\mu^2}{\|\bw\|^2}}.\numberthis\label[ineq]{eq:lower_bound_minus_wbarTopxbar}\] Next, we take the inner product of some element in $\partial \varphi(\bx)$ with $-\bar{\bw}$ (as defined in \cref{{eq:last_case_one_diff_exp}}) and plug in \cref{eq:lower_bound_minus_wbarTopxbar} to obtain: \[-\tfrac{1}{32}\bar{\bw}^\top \bar{\bx} + v - \tfrac{v}{2}\cdot\bar{\bw}^\top (\overline{\bx+\bw})\geq \tfrac{1}{32}\sqrt{1-\tfrac{100\mu^2}{\|\bw\|^2}} >\tfrac{1}{33}, \] where the final step follows from our choice of $\bw$ such that $\|\bw\|=1000\mu$. 
   \item $x_d=0,~\bx\notin \{\bzero,-\bw\}, ~\|\bx+\bw\|\geq 10\mu, ~\inner{\bar{\bw},\overline{\bx+\bw}}\in[\tfrac{1}{2},\tfrac{1}{2}+\tfrac{\mu}{\norm{\bx+\bw}}]:$ As in \cref{{eq:last_case_one_diff_exp}} in the previous case, we have for $0\leq v\leq 1$ that \begin{align*} 
   \partial \varphi(\bx) &= \left\{ \tfrac{1}{32} \overline{\bx_{1:d-1}} + \lambda \be_d  - v\cdot (\bar{\bw} 
 - \tfrac{1}{2}\overline{\bx+\bw})~|~\lambda\in\partial h(x^*), v\in [0, 1]\right\},\numberthis\label{eq:last_case_really_the_last_case_one_diff_exp}\\ 
 &= \left\{\lambda \be_d  + \left(\tfrac{1}{32\|\bx\|} + \tfrac{v}{2\|\bx+\bw\|}\right)\bx + \left(\tfrac{v}{2\|\bx+\bw\|} - \tfrac{v}{\|\bw\|}\right)\bw  ~|~\lambda\in\partial h(x^*), v\in [0, 1]\right\}.\end{align*} Denote $\bx= \bx_{|} + \bx_{\perp}$, where $\bx_{\perp} = (\mathbf{I} - \bar{\bw}\bar{\bw}^\top)\bx$ is the orthogonal projection of $\bx$ onto $\mathrm{span}(\bw)^\perp$, and $\bx_{|}\in \mathrm{span}(\bw)$. Recall also that $x_d=0$ (by assumption in this case) and $w_d =0$ (by our choice of $\bw$). For any $\mathbf{g}\in \partial\varphi(\bx)$, we can therefore write, for some scalar $\alpha$, that \begin{align*}
     \|\mathbf{g}\| &\geq \|\left( \tfrac{1}{32\|\bx\|} + \tfrac{v}{2\|\bx+\bw\|} \right) \bx + \left( \tfrac{v}{2\|\bx+\bw\|} - \tfrac{v}{\|\bw\|}\right) \bw\|\\
     &= \| \left( \tfrac{1}{32\|\bx\|} + \tfrac{v}{2\|\bx+\bw\|} \right) \bx_{\perp} + \alpha \bw\|\\
     &\geq \|\left( \tfrac{1}{32\|\bx\|} + \tfrac{v}{2\|\bx+\bw\|}\right) \bx_{\perp}\|\\ 
     &\geq \tfrac{\|\bx_{\perp}\|}{32\|\bx\|}, \numberthis\label[ineq]{eq:last_case_g_norm_lb_intermediate}
 \end{align*}  where the first step is by projecting onto $\mathrm{span}(\be_d)^\perp$; the second step is by splitting $\bx$ into $\bx_{|}+\bx_{\perp}$ and absorbing $\bx_{|}$ into the term written as a multiple of $\bw$ (valid since $\bx_{|}\in \mathrm{span}(\bw)$); the third step is because $\bx_{\perp}\perp \bw$, and so the norm of their sum is at least as large as each of them; the fourth step is by $v\geq0$. Further, since $\mathbf{I}-\bar{\bw}\bar{\bw}^\top$ is an orthogonal projection, we have \[\|\bx_{\perp}\|^2 = \inner{\bx, (\mathbf{I}-\bar{\bw}\bar{\bw}^\top)^2 \bx} = \inner{\bx, (\mathbf{I}-\bar{\bw}\bar{\bw}^\top) \bx} = \|\bx\|^2(1-\inner{\bar{\bw}, \bar{\bx}}^2).\] Plugging into \cref{eq:last_case_g_norm_lb_intermediate} yields $\|\mathbf{g}\|\geq \tfrac{1}{32}\sqrt{1-\inner{\bar{\bw}, \bar{\bx}}^2}$, where the square root operation is valid because $|\inner{\bar{\bw}, \bar{\bx}}|\leq 1$.  Now, suppose that there exists a $\mathbf{g}\in \partial\varphi(\bx)$ with $\|\mathbf{g}\|\leq \epsilon\leq \tfrac{1}{32}$ (note that if $\|\mathbf{g}\|\geq \tfrac{1}{32}$ for all $\mathbf{g}\in \partial\varphi(\bx)$, then we are done). It then follows that \[\tfrac{1}{32}\sqrt{1-\inner{\bar{\bw}, \bar{\bx}}^2} \leq \epsilon.\numberthis\label[ineq]{eq:final_case_last_but_one_ineq}\] Next, the assumed range implies \[\tfrac{1}{2}\leq \inner{\bar{\bw},\overline{\bx+\bw}} \leq \tfrac{3}{5}.\numberthis\label[ineq]{eq:assumption_implies_sandwich_bound}\] If $\inner{\bar{\bw}, \bar{\bx}}\leq 0$, then by \cref{eq:final_case_last_but_one_ineq}, it must be that $\inner{\bar{\bw}, \bar{\bx}}\leq -\sqrt{1-1024\epsilon^2}$. Consider any  $\mathbf{u}\in \partial \varphi(\bx)$; plugging  \cref{eq:assumption_implies_sandwich_bound} into \cref{{eq:last_case_really_the_last_case_one_diff_exp}}, we have \[\inner{\mathbf{u}, \bar{\bw}} = \tfrac{1}{32}\inner{\bar{\bw}, \bar{\bx}} - v\cdot(1-\tfrac{1}{2}\cdot\tfrac{3}{5})\leq -\tfrac{1}{32}\sqrt{1-1024\epsilon^2},\] which implies that $\|\mathbf{u}\|\geq \tfrac{1}{32}\sqrt{1-1024\epsilon^2}$ for any $\mathbf{u}\in \partial\varphi(\bx)$. Thus, if $\|\mathbf{u}\|\leq\epsilon$, then chaining the two inequalities yields $\epsilon\geq \tfrac{1}{32}\sqrt{1-1024\epsilon^2}$. This implies that $\epsilon\geq \tfrac{1}{\sqrt{2048}}$. On the other hand, if $\inner{\bar{\bw}, \bar{\bx}} \geq 0,$ then combining with \cref{eq:final_case_last_but_one_ineq} gives $\inner{\bar{\bw}, \bar{\bx}}\geq \sqrt{1-1024\epsilon^2}$. Hence, \[ \inner{\bar{\bw}, \bx+\bw} = \inner{\bar{\bw}, \bx} + \|\bw\| \geq \|\bx\|\sqrt{1-1024\epsilon^2} + \|\bw\|\geq \sqrt{1-1024\epsilon^2}\|\bx+\bw\|. \] If $\epsilon<\tfrac{1}{50},$ then $\tfrac{3}{5}\geq \inner{\bar{\bw}, \overline{\bw+\bx}} > \sqrt{1-\tfrac{1024}{2500}}$, which is a contradiction. Thus, combining both the cases, we see that the lower bound must be at least $\min\left(\tfrac{1}{\sqrt{2048}}, \tfrac{1}{50}\right).$
\end{itemizec}

From the above analysis, we conclude that $\varphi(\bx)=h(\bx+\bx^*)-\sigma_{\mu}\left(\inner{\bar{\bw},\bx+\bw}-\tfrac{1}{2}\norm{\bx+\bw}\right)$ has no $c$-stationary points for  $c\leq \tfrac{1}{100}$. We now show that, {for $c = \tfrac{1}{100}$}, any $c$-stationary point of $f(\bx)=\max\{\varphi(\bx-\bx^*),0\}$ (which matches the definition of $f$ in \cref{{eq:def_fwu}} combined with \cref{rem:choice_of_w}) satisfies $f(\bx)=0$.  Indeed, if there existed  $\bx$ with 
$\norm{\bar\partial f(\bx)}\leq c$ and $f(\bx)>0$, then we note by the latter that $f(\cdot)=\varphi(\cdot-\bx^*)$ in an open neighborhood of $\bx$, thus 
$\norm{\bar\partial \varphi(\bx-\bx^*)}=\norm{\bar\partial f(\bx)}\leq c$, which is a contradiction by our earlier claim on $c$-stationarity of $\varphi$.

~\paragraph{Proof of \Cref{lem: choice of w iii}.}

We denote by $(\bx_t^{h})_{t=1}^{T}$ the (possibly random) iterates produced by $\Acal$ when applied to $h$.
We will first show that for some fixed $\bw\in\reals^d:$
\begin{align*} 
    \Pr_{\Acal}\left[
    \min_{t\in[T]}\norm{\bx_t^{h}-\bx^*}\geq \rho\text{~~and~~}\max_{t\in[T]}\inner{\bar\bw,\overline{\bx_t^{h}-\bx^*}}<\frac{1}{3}
    \right]&\geq 1-2\gamma~.\numberthis\label[ineq]{eq: good event}
\end{align*}
To see why, recall that by \cref{lem: d>=2 hard} we know that
$\Pr_{\Acal}[\min_{t\in[T]}\norm{\bx_t^{h}-\bx^*}\geq \rho]\geq1-\gamma$.
Furthermore, letting $\bw\in\reals^d$ be a random vector that satisfies $\bw_{1:d-1}\sim\mathrm{Unif}(\tfrac{\rho}{99}\cdot\Sbb^{d-2}),~\Pr[w_d=0]=1$, and noting that $(\bx_t^{h})_{t=1}^{T},\bx^*$ are independent of $\bw$ and that 
$\bar{\bw}_{1:d-1}\sim\mathrm{Unif}(\Sbb^{d-2})$, applying a standard tail bound on the inner product of a uniformly random unit vector (cf. \citep[Lemma 2.2]{ball1997elementary}) we get
\[
\Pr_{\bw}\left[\max_{t\in[T]}\inner{\bar\bw,\overline{\bx_t^{h}-\bx^*}}\geq\frac{1}{3}\right]
=\Pr_{\bw}\left[\max_{t\in[T]}\inner{\bar\bw_{1:d-1},\overline{(\bx_t^{h}-\bx^*)}_{1:d-1}}\geq\frac{1}{3}\right]
\leq T\exp\left(-d/36\right)=\gamma~.
\]
By the union bound, we see that \Cref{eq: good event} holds with probability at least $1-2\gamma$ over the joint probability of $\bw,\Acal$, thus (via the probabilistic method argument) there exists some fixed $\bw$ for which it holds over the randomness of $\Acal$. We therefore fix $\bw$ so that \Cref{eq: good event} holds, and assume the high probability event indeed occurs. We aim to show that under this event, for all $t\in[T]:~f(\bx_t)\geq1$, which will then conclude the proof.
Indeed, under this event, we see that
\[
\max_{t\in[T]}\inner{\bar\bw,\overline{\bx_t^{h}-\bx^*}}
<\frac{1}{3}<\frac{1}{2}-\frac{1}{66}
\leq \frac{1}{2}-\frac{\rho}{66\norm{\bx_t^{h}-\bx^*}}~.
\]
Further note that for any $\bx\neq\bx^*$, if $\inner{\bar\bw,\overline{\bx-\bx^*}}<\frac{1}{2}-\frac{\rho}{66\norm{\bx_t^{h}-\bx^*}}$ then 
since $\norm{\bw}=\tfrac{\rho}{99}$
a straightforward calculation yields $\inner{\bar{\bw},\bx-\bx^*+\bw}-\frac{1}{2}\norm{\bx-\bx^*+\bw}< 0$. As this is an open condition with respect to $\bx$, by construction of $f$ this implies that $f(\cdot)=h(\cdot)$ in a neighborhood of $\bx$.
We therefore get that for all $t\in[T]:h\equiv f$ in a neighborhood of $\bx_t^h$, so in particular we see that $\bx_t^h=\bx_t$, namely applying $\Acal$ to $h$ results in the same iterate sequence as if the algorithm were applied to $f$. Thus
\[
\min_{t\in[T]}f(\bx_t)
=\min_{t\in[T]}f(\bx_t^h)
=\min_{t\in[T]}h(\bx_t^h)\geq 1~.
\]

\end{proof}

\section*{Acknowledgements}

We would like to thank Ali Jadbabaie, Dmitriy Drusvyatskiy, Damek Davis, and Hadi Daneshmand 
for insightful discussions.
GK is supported by an Azrieli Foundation graduate fellowship.
SP is supported by ONR N00014-23-1-2299 (via Ali Jadbabaie) and  NSF CCF-2112665 (via Suvrit Sra). 
GK and OS are supported in part by European Research Council (ERC) grant 754705.

\clearpage
\printbibliography

\newpage

\end{document}